\documentclass[11pt]{article}

\usepackage{amsmath,amssymb,amsthm,amsfonts,enumerate,graphicx}
\usepackage{amssymb, amsfonts,color}
\usepackage{amsmath}
\usepackage{latexsym}
\usepackage{mathrsfs}
\usepackage{graphicx}
\usepackage{enumerate}

\usepackage[backref=page]{hyperref}
\usepackage{hyperref}
\newtheorem{theorem}{Theorem}[section]
\newtheorem{lemma}[theorem]{Lemma}
\newtheorem{proposition}[theorem]{Proposition}

\newtheorem{definition}{Definition}[section]
\newtheorem{remark}{Remark}[section]

\newtheorem{example}[theorem]{Example}
\numberwithin{equation}{section}
\renewcommand{\baselinestretch}{1.0}

\begin{document}

\title{Spectrality of product-form self-similar measures and tiles
\footnotetext{$^*$~Corresponding author.}
\footnotetext{ 2020  {\it Math Subject Classifications}: Primary 28A25, 28A80; Secondary 42C05, 46C05.}
\footnotetext{J.C. Liu is supported  by the Hunan Provincial Natural Science Foundation, Grant No.2024JJ3023 and J. Zheng is supported  by the China Postdoctoral Science Foundation under Grant Number 2025M783128.}
\footnotetext{Emails: jcliu@hunnu.edu.cn, wjj2021hnsd@163.com, zhengjiamath@163.com}}
\author{Jing-Cheng Liu$^1$, Jia-Jie Wang$^{*,1}$, Jia Zheng$^2$ \\
   \small {1.} Key Laboratory of Computing and Stochastic Mathematics \\
   \small  (Ministry of Education), School of Mathematics and Statistics,\\
   \small Hunan Normal University, Changsha, Hunan 410081, P.R. China\\
   \small {2.} School of Mathematics and Statistics, \\
   \small Wuhan University, Wuhan Hubei 430072, P.R. China }

\date{\today}
\maketitle

\begin{abstract}
This paper studies the Fourier properties of self-similar measures and tiles generated by digit sets of product-form. Let  $0 <\rho <1$ be a real number and  let $D$  be the direct sum of two consecutive integer sets:
$$D=\{0,1,\cdots,N-1\}\oplus m\{0,1,\cdots, L-1\},$$
where  $N, m, L \in \mathbb{N}^{*}$ with 
$N, L \geq 2$.
The pair $(\rho,D)$ determines the self-similar iterated function system (IFS) $	\{\phi_d(\cdot)=\rho(\cdot+d)\}_{d \in D}$. Let $\mu_{\rho,D}$ and $T$ be the associated self-similar measure and self-similar set, respectively.
We first prove that $L^2(\mu_{\rho,D})$ admits an exponential orthonormal basis if and only if $\rho^{-1}=p\in\mathbb{N}$ satisfies $N\mid p$, $L\mid p$ and $N\mid \frac{m}{\gcd(m,p^d)}$,
where
$$d=\max\left\{i:\gcd\left(\frac{mL}{\gcd(mL,p^i)},L\right)\neq 1,i\in\mathbb{N}\right\}.$$  
This result extends a series of previous studies, including the cases where $N,L$ are primes [An-Wang, {\it J. Funct. Anal.}, 2021] and $N=L$ [Liu-Peng-Wu, {\it J. Math. Anal. Appl.}, 2019]. Furthermore, 
in the context of the Fuglede conjecture, we show that when $\rho^{-1} =\#D= NL$, the space $L^2(\chi_T dx)$  admits an exponential orthonormal basis if and only if $T$  is a translation tile of $\mathbb{R}$.


\end{abstract}

\renewcommand{\baselinestretch}{1.0}

\section{Introduction}

Let $\mu$ be a Borel probability measure with compact support on $\mathbb{R}^n$. We say that $\mu$ is a \emph{spectral measure} if there exists a countable subset $\Lambda$ of $\mathbb{R}^n$ such that the set of exponential functions
\[
\mathcal{E}_{\Lambda}:=\big\{e^{2\pi i \langle \lambda, x \rangle} : \lambda\in \Lambda\big\}
\]
forms an orthonormal basis for $L^2(\mu)$. In this case, we call $\Lambda$ a \emph{spectrum} for $\mu$. If a spectral measure $\mu$ is the normalized Lebesgue measure restricted to a Lebesgue measurable set $\Omega\subseteq\mathbb{R}^n$ of positive and finite measure, then we say $\Omega$ is a \emph{spectral set}. It is well known in classical Fourier analysis that the unit cube $[0,1]^n$ is a spectral set with spectrum $\mathbb{Z}^n$. The study of spectral sets is closely related to the celebrated Fuglede conjecture \cite{Fuglede1974}, which states that:
\begin{center}
	\textit{$\Omega$ is a spectral set if and only if $\Omega$ tiles $\mathbb{R}^n$ by translations.}
\end{center}
In his seminal paper, Fuglede showed that certain planar domains, such as triangles and disks, are not spectral. The conjecture was later disproved in both directions for dimensions $n \ge 3$ by Tao \cite{Tao2003}, and by Kolountzakis and Matolcsi \cite{Kol2006',Matolcsi2005}. At present, the conjecture remains completely open in dimensions $n = 1$ and $n = 2$. Although Fuglede's conjecture fails in higher dimensions, positive results have been established in several special cases (see \cite{Laba2001,Fan2019,IMP_2017,Shi_2019} and the references therein). In particular, Lev and Matolcsi \cite{Lev2022} proved that Fuglede's conjecture holds true for all convex domains. Moreover, they demonstrated that every spectral set is a \emph{weak tile}, which serves as a measure-theoretic relaxation of the classical notion of translation tiling.

While spectral sets have been extensively studied, achieving a complete classification of spectral measures remains challenging. It was shown in \cite{HLL_2013} that any spectral measure $\mu$ must be of pure type; namely, $\mu$ is either discrete, singularly continuous, or absolutely continuous with respect to the Lebesgue measure. In the discrete case, the spectrality of $\mu$ is closely related to integer tilings \cite{HLL_2013}. In the absolutely continuous case, it is necessary \cite{DL_2014} that $\mu$ coincides with the Lebesgue measure on some measurable set $\Omega$. Consequently, the study of absolutely continuous spectral measures naturally reduces to the original Fuglede conjecture. This reduction  motivates a separate, dedicated investigation into singular spectral measures. 

The first non-atomic singular spectral measure, the middle-fourth Cantor measure, was elegantly constructed by Jorgensen and Pedersen in 1998 \cite{Jorgensen-Pedersen_1998}. Fourier analysis on singular measures exhibits phenomena that are different from those in the classical absolutely continuous setting. A key distinction lies in the distribution of the spectrum: unlike absolutely continuous measures, the spectrum of a singular measure can be arbitrarily sparse \cite{AL_2023}. Despite this distinguish, Strichartz \cite{Strichartz_2000,Strichartz_2006} proved that for certain spectra of such Cantor measures, the associated Fourier series possess even better convergence properties than those on the standard unit interval.

Accordingly, the focus of this paper is to investigate the spectrality of singular measures, followed by a further exploration of the spectral and tiling properties of the associated self-similar sets.

\subsection{Spectrality of self-similar measure with product-form digit sets}
Let us recall the standard definitions  in fractal geometric \cite{Falconer_1990}.
Let $\{\phi_d(x)\}_{d\in D}$ be an {\it iterated function system} (IFS) defined by
	\begin{equation*}\label {eq(1.1)}
		\phi_d(x)=\rho(x+d),
	\end{equation*}
	where $ x\in \mathbb{R}, 0<\rho<1$ and  $D\subset \mathbb{R}$ is a finite set.
	The {\it self-similar measure} is the unique probability measure $\mu:=\mu_{\rho,D}$ satisfying
	\begin{equation}\label{1.1}
		\mu=\frac{1}{\# D}\sum_{d\in D}\mu\circ\phi_d^{-1},
	\end{equation}
	where $\#D$ is the cardinality of the digit set $D$. Such a measure $\mu_{\rho,D}$ is supported on the {\it self-similar set} (or {\it attractor}) $ T(\rho,D)$, which is the unique  nonempty compact set satisfying
	\begin{equation*}
		T(\rho,D)=\bigcup_{d\in D}\phi_d(T(\rho,D)).
	\end{equation*}
	Moreover, $T(\rho,D)$ can be expressed by the following radix expansion
	\begin{align}\label{eqa}
		T(\rho,D)=\left\{\sum_{j=1}^{\infty}\rho^{j}d:d\in D \ {\rm for \ all} \ j\in\mathbb{N}\right\}.
	\end{align}
 Following the work of Jorgensen and Pedersen \cite{Jorgensen-Pedersen_1998}, Hu and Lau \cite{Lau2008} studied infinite orthogonal sets of exponential functions for the Bernoulli convolution measures $\mu_{\rho,\{0,1\}}$ with $0<\rho<1$.
 Dai \cite{Dai2012} gave a complete description of the spectrality of $\mu_{\rho,\{0,1\}}$. Later, Dai, He and Lau extended $\mu_{\rho,\{0,1\}}$ to the $N$-Bernoulli convolution measures $\mu_{\rho,D}$, where $D=\{0, 1, \cdots, N-1\}$ is a consecutive digit set for any positive integer $N\geq 2$. They \cite{Dai-He-Lau_2014}  proved that $\mu_{\rho,D}$ is a spectral measure if and only if $\rho^{-1}=p$ is an integer and $N$ divides $p$.
Indeed, a key observation is that the spectrality of $\mu_{\rho,D}$ can be reformulated within the setting of Hadamard triples. In some sense, the spectrality of the measure is closely to the existence of the Hadamard triple.

\begin{definition}\label{2.5}
	Let $R \in M_n(\mathbb{Z})$ be an expanding matrix\footnote{ $R\in GL_n(\mathbb{R})$ is an expanding matrix means that all the eigenvalues of $R$ are larger than one in module.} and $D,L \subset \mathbb{Z}^n$ be two finite digit sets with the same cardinality, $i.e.,\#D=\#L$. We say that $(R,D,L)$ is a Hadamard triple (or $(R,D)$ is an admissible pair) if the matrix
	\begin{align}
		H=\frac{1}{\sqrt{\#D}} \left( e^{2 \pi i \langle M^{-1}d,l \rangle} \right)_{d \in D,l \in L}   \notag
	\end{align}
	is unitary, i.e., $H^{\ast}H=I_n$, where $H^{\ast}$ means the transposed conjugate of $H$ and $I_n$ denotes the $n\times n$ identity matrix.
\end{definition}

Jorgensen and Pedersen \cite{Jorgensen-Pedersen_1998} proved that if $(R,D,L)$ is a Hadamard triple, then
$\mathcal{E}_{\Lambda(R,L)}$ admits an infinite orthogonal system in $L^2(\mu_{R^{-1},D})$, where
\[
\Lambda(R,L)=\left\{\sum_{j=0}^{k-1} R^{*j}l_j:k\geq1,  l_j\in L \right\}.
\]
Dutkay and Jorgensen \cite{DJ_2007,DJ_2009} conjectured that $\mu_{R^{-1},D}$ is a spectral measure whenever $(R,D)$ is admissible.
This conjecture was first proved on $\mathbb{R}$ by \L aba and Wang \cite{LaW_2002}.
In higher dimensions, the situation is more complex.
Partial results were obtained under additional assumptions in \cite{DJ_2007,Strichartz_2000}.
This conjecture was finally solved by Dutkay, Haussermann and Lai \cite{Dut-Hau-Lai2019}.
Although Hadamard triples are sufficient to generate spectral self-similar measures, not all spectral self-similar measures can be generated by Hardamard triples.

\begin{example}[\cite{DJ_2009_2}]\label{examp1.2}
	Let \( R = 4 \) and \( D = \{0, 1, 8, 9\} \). Then $\mu_{4^{-1},D}$ coincides with the Lebesgue measure restricted to \( T = [0,1] \cup [2,3] \), and  $T$ is spectral with a spectrum \( \mathbb{Z} + \{0,1/4\} \). However, since the elements of $D$ are not distinct modulo 4, there does not exist any integer $L$ such that \( (R, D, L) \) forms a Hadamard triple.
\end{example}

As illustrated in Example \ref{examp1.2}, although the set \( D=\{0,1,8,9\} \) falls outside the scope of classical Hadamard triples, it exhibits a hidden structure and can be decomposed into
\[
D=\{0,1\}\oplus 2^{3}\{0,1\}.
\]
Here, the symbol $\oplus$ denotes the direct sum, meaning that each element in $D$ has a unique representation. This representation places $D$ within a broader family of digit sets characterized by a  direct sum structure. 
For convenience, in what follows, we use the symbol $D_n$ to denote the consecutive digit set $\{0,1,\dots,n-1\}$ for any positive integer $n \in \mathbb{N}^*$. Under this setting, Liu et al. \cite{Liu-Peng-Wu_2019} investigated the spectrality of self-similar measures $\mu_{\rho,D}$ generated by $\rho^{-1}=N^{q}$ and the specific direct sum structure digit set
$
	D=D_N\oplus N^{p}D_N,
$
where $q,p \ge 1$ and $N \ge 2$ are integers. They proved that $\mu_{\rho,D}$ is a spectral measure if and only if $q$ does not divide $p$. 

Recently, An and Wang \cite[Theorem 1.7]{An2021} advanced this direction by characterizing a class of self-similar spectral measures arising from the so-called strict product-form digit sets. Recall that a digit set $D$ with $\#D=b$ is said to be a \emph{strict product-form digit set with respect to $b$} if $D=E_0\oplus b^{\ell} E_1$, where $ \ell\geq 0$ and $E_0\oplus E_1=\{0,1,2,\dots, b-1\}$. For such sets, they proved the following result:
\begin{theorem}[\cite{An2021}]\label{A}
	Let \( 0 < \rho < 1 \) and let \( p, q \) be primes. Let \( D \) be a strict product-form digit set with respect to \( pq \). Then \( \mu_{\rho,D} \) is a spectral measure if and only if \( \rho^{-1} = p^{l_1}q^{l_2}t_1 \) and \( D \) is a CPF digit set\footnote{A finite set \( \mathcal{D} \subset \mathbb{Z} \) is called a complementing product-form (CPF) digit set with respect to \( b \) if  
		$ \mathcal{D} = \mathcal{D}_0 \oplus b^{l_1} \mathcal{D}_1 \oplus \cdots \oplus b^{l_k} \mathcal{D}_k $
		and \( \tilde{\mathcal{D}} = \mathcal{D}_0 \oplus \mathcal{D}_1 \oplus \cdots \oplus \mathcal{D}_k \) tiles \( \mathbb{Z}_b \), where \( 0 \leq l_1 \leq \cdots \leq l_k \).} respect to \( p^{l_1}q^{l_2} \), where \( l_1, l_2, t_1 \) are positive integers and \( \gcd(t_1, pq) = 1 \).
\end{theorem}
While strict product-form digit sets have been well studied, restricting the parameters $p,q$ to primes limits their generality. Inspired by this, we consider a class of more general product-form digit sets: 
\begin{equation}\label{eq1.3}
	D = D_N \oplus mD_L,
\end{equation}
where $N, m, L \in \mathbb{N}^{*}$ with $N, L \geq 2$. To ensure our definition is well-defined, we require $m \ge N$. 
We point out that the terminology ``product-form'' was originally coined by Lagarias and Wang \cite{LW_1996_2} in the context of self-affine tiles. In our work, we adopt the term \emph{product-form} to emphasize that our model inherits the structural essence (the direct sum of scaled blocks). 
The motivation for studying the product-form digit sets defined in \eqref{eq1.3} is twofold. 
	First, this class of digit sets naturally generalizes the strict product-form digit sets. For instance, the digit set studied by An and Wang \cite{An2021} can be expressed in our work as:
	\begin{equation}\label{eq1.4}
		D=D_p\oplus (pq)^l p D_q,\quad l\geq 0, \text{ where } p,q \text{ are primes}.
	\end{equation}
	Second, these product-form digit sets provide a useful perspective for considering spectrality problems in other iterated function systems (IFSs), such as those with alternating contraction ratios. For example, Wu \cite{Wu2024} studied the spectrality of measures generated by the alternating contractive IFS $\{ \phi_d(x)=(-\rho)^d(x+d) \}_{d\in D_{2N}}$. An  observation in  \cite[Proposition~4.2]{Wu2024} is that this problem reduces to analyzing the measure $\mu_{\rho,D}$ defined by \eqref{1.1}, where
	$
	D = D_N \oplus \left( N - \frac{1+\rho}{2} \right) D_2 .
	$
	This shows that such spectrality problems can  be reduced to the study of product-form digit sets. 
	Consequently,  characterizing the spectrality of measures associated with product-form digit sets is not merely a  generalization, but rather a step toward understanding the spectral properties of a much broader class of fractal measures.

Note that when $N=1$ or $L=1$, the spectrality of $\mu_{\rho,D}$ follows directly from the results of Dai, He and Lau \cite{Dai-He-Lau_2014}. Therefore, throughout this paper,  we restrict our attention to the case where $N \ge 2$ and $L \ge 2$. Our main result is stated as follows.

\begin{theorem}\label{1.2}
	Let $\mu_{\rho,D}$ be the self-similar measure defined by \eqref{1.1}, where $0<\rho<1$ and $D$ is given by \eqref{eq1.3}. Then $\mu_{\rho,D}$ is a spectral measure if and only if $\rho^{-1}=p\in\mathbb{N}$ satisfies
	\[
	N\mid p,\quad L\mid p,\quad \text{and}\quad N\mid \frac{m}{\gcd(m,p^d)},
	\]	
	where
	\begin{equation}\label{2.11}
		d=\max\left\{i \in \mathbb{N} : \gcd\left(\frac{mL}{\gcd(mL,p^i)},L\right)\neq 1\right\}.
	\end{equation}
\end{theorem}

\noindent{\bf Strategy of the proof and remarks}. 
	\begin{enumerate}[\rm(1)]
		
		\item 
		The main difficulty in our proof arises from the independence of the parameters $N, L,$ and $m$, which leads to complicated relations among the zero sets of the associated Fourier transforms. To overcome this, we give an equivalent definition of $d$ (see \eqref{eq3.1.0} below) for the theoretical proofs. While the explicit formula in \eqref{2.11} is highly convenient for concrete applications (see, for instance, Examples 1.3 and 1.4), the equivalent formulation in \eqref{eq3.1.0} captures the essential arithmetic relations among the parameters much better.
		
		\item 
		To prove the necessity of the condition $N \mid \frac{m}{\gcd(m,p^d)}$, we proceed by contradiction. Assuming $N \nmid \frac{m}{\gcd(m,p^d)}$, we show that any orthogonal set $\Lambda$ for $L^2(\mu_{\rho,D})$ must be incomplete. We achieve this by constructing a nonzero function $f \in L^2(\mu_{\rho,D})$ that is orthogonal to $e^{2\pi i \lambda x}$ for all $\lambda \in \Lambda$. While the initial motivation for this construction stems from An, He and Lai \cite{An-He-Lai2023}, we adopt a distinct, purely measure-theoretic method. Specifically, we use the Radon-Nikodym theorem to guarantee the existence and $\mu_{\rho,D}$-almost everywhere uniqueness of the function $f$. 
		
		\item 
	   Applying Theorem \ref{1.2} to the digit set $D$ defined in \eqref{eq1.4}, we can deduce that $\mu_{\rho, D}$ is a spectral measure if and only if $\rho^{-1}=p^{l_1}q^{l_2}t \in \mathbb{N}^{*}$ with $l_1,l_2 \geq 1$ and $\gcd(t,pq)=1$, and $l\geq l_1\lfloor\frac{l}{l_2}\rfloor$. This characterization seems more intuitive than the conditions given in Theorem \ref{A}. 
			In fact, Theorem \ref{1.2} completely resolves the spectrality problem for self-similar measures generated by strict product-form digit sets consisting of two direct sum components.

		\item 
		Finally, we remark that the conditions $N \mid p$, $L \mid p$, and $N \mid \frac{m}{\gcd(m,p^d)}$ imply $NL \mid p$, which essentially forces $p \geq \# D$ (as detailed in Remark~\ref{rem3.1}). In most nontrivial cases, the condition $p > \# D$ indicates that the associated self-similar measure is singular. The case where $p < \# D$ is significantly more complicated due to geometric overlaps among the sets $T+d$ for $d \in D$. The critical boundary case $p = \# D$ is intimately connected to self-similar tiles and spectral sets, and will be explored in depth in Section \ref{sec4}.
		
\end{enumerate}

To gain some insight into our results, we present two examples below. 
\begin{example}
	Let $\mu_{\rho,D}$ be the self-similar measure defined by \eqref{1.1}, where
	$\rho^{-1}=72$ and $D=D_{12}\oplus mD_4$ with some $m\in \mathbb{N}^*$. Then $\mu_{\rho,D}$ is not a spectral measure for any $m\in\mathbb{N}^*$.
\end{example}

\begin{proof}
	Suppose, to the contrary, that $\mu_{\rho,D}$ is a spectral measure for some
	$m=2^\tau m'$ with $\tau\in\mathbb{N}^*$ and $m'\in 2\mathbb{Z}+1$. Let $d$ be defined by \eqref{2.11}. Then we have
$$\gcd\left(\frac{2^{\tau+2}m'}{\gcd(2^{\tau+2}m',2^{3(d+1)}3^{2(d+1)})},4\right)=1,$$
which yields that $\tau\leq 3d+1$.
	According to Theorem~\ref{1.2}, it is necessary that
\[4 \mid \frac{2^{\tau}m'}{\gcd(2^{\tau}m',2^{3d}3^{2d})}.
	\]
	This condition implies $\tau-3d\ge 2$, which contradicts the fact that
	$\tau\leq 3d+1$. Therefore, there exists no
	$m\in\mathbb{N}^*$ such that $\mu_{\rho,D}$ is a spectral measure.
\end{proof}

\begin{example}
	Let $\mu_{\rho,D}$ be the self-similar measure defined by \eqref{1.1}, where
	$\rho^{-1}=144$ and $D=D_{12}\oplus mD_4$ with some $m\in \mathbb{N}^*$. Then $\mu_{\rho,D}$ is a
	spectral measure if and only if
	$
	m=2^{\tau_1}3^{\tau_2}m'
	$  satisfies
\[
	\tau_1=4k+2 \quad\text{and}\quad \tau_2\ge 2k+1
	\]
 for some nonnegative integer $k$ and $m' \in \mathbb{N}^*$ with $\gcd(m',6)=1$.
\end{example}

\begin{proof}
Note that $p:=\rho^{-1}=2^4\cdot 3^2$
 and $N:=12=2^2\cdot 3$. By using Theorem \ref{1.2}, it follows that $\mu_{p^{-1},D}$ is a spectral measure if and only if $12\mid \frac{m}{\gcd(m,p^d)}$, where $d$ is defined by \eqref{2.11}.
 Then the condition $12\mid \frac{m}{\gcd(m,p^d)}$ gives that
\begin{equation}\label{2.13}
\tau_1-4d\geq2\quad {\rm{and}}\quad \tau_2-2 d\geq 1.
\end{equation}
On the other hand, the definition of $d$ in \eqref{2.11} forces that $\tau_1+2\leq 4(d+1)$. Combining this and \eqref{2.13} gives that $\tau_1=4d+2$ and $\tau_2\geq 2d+1$.
\end{proof}

\subsection{Fuglede conjecture on self-similar sets with product-form digit set}

In this section, we investigate the spectral and tiling properties of self-similar sets generated by product-form digit sets, which are closely connected to the one-dimensional Fuglede conjecture. 

Recall that if $\rho^{-1}=\#D$ and the attractor $T:=T(\rho,D)$ defined by \eqref{eqa} has a nonempty interior, then $T$ is known to tile $\mathbb{R}$ by translations \cite{LW_1996}; such a set is called a \emph{self-similar tile}. Within the context of Fuglede's conjecture, it is highly natural to ask whether such a self-similar tile is necessarily a spectral set. This question lies at the heart of the ``tile implies spectral'' direction of the conjecture in $\mathbb{R}$. Consequently, the spectrality of self-similar (and, more generally, self-affine) tiles has been the subject of extensive study in the literature.

In dimension one, Kenyon \cite{Ken1992} studied tile digit sets using number-theoretic methods. Later, Coven and Meyerowitz \cite{CM} famously introduced the (T1) and (T2) conditions, which deeply influenced the spectral characterization of finite sets and self-similar tiles. Lagarias and Wang \cite{LW_1996,LW_1996_2} formalized the notions of product-form and modulo product-form digit sets, while the spectrality of their corresponding self-similar tiles was subsequently investigated by Fu et al.~\cite{Fu-He-Lau2015} and Lai et al.~\cite{Lai2017}.
In higher dimensions, An and Lau \cite{An2019} examined self-affine tiles generated by the expanding matrix $pI_{2}$, where $p$ is a prime. More recently, Li and Rao \cite{LR_2025} extended several characterization results to self-affine tiles in $\mathbb{R}^{n}$ and obtained specific spectral properties for self-similar tiles in $\mathbb{R}$. Furthermore, Chen et al.~\cite{Chen-Liu-Zheng2024} proved that a generalized Sierpi\'{n}ski-type set is a spectral set if and only if it is a self-affine tile.

As a direct consequence of Theorem \ref{1.2}, we establish the following equivalence between the tiling and spectral properties for self-similar sets generated by product-form digit sets, thereby verifying Fuglede's conjecture for this specific class of fractals.

\begin{theorem}\label{1.3}
	Let $T(\rho,D)$ be the self-similar set defined by \eqref{1.1}, where $\rho^{-1}=LN$ and $D$ is given by \eqref{eq1.3}.
	Then $T(\rho,D)$ is a spectral set if and only if $T(\rho,D)$ is a translation tile of $\mathbb{R}$.
\end{theorem}

The remainder of the paper is structured as follows.
In Section~\ref{sec2}, we provide some basic definitions and preliminaries concerning self-similar measures.
In Section~\ref{sec3},  we prove Theorem~\ref{1.2}, which establishes  a necessary and sufficient condition for the spectrality of $\mu_{\rho,D}$.
Finally, in Section~\ref{sec4}, we study the tiling and spectral properties of the self-similar set $T(\rho,D)$ and present the proof of Theorem~\ref{1.3}.

\section{Preliminaries}\label{sec2}

In this section, we recall some basic definitions and useful conclusions related to self-similar measures. Assume that $\mu$ is a probability measure with compact support on $\mathbb{R}$. The Fourier transform of $\mu$ is defined as
$$\hat{\mu}(\xi)=\int_{\mathbb{R}} e^{2 \pi i \langle \xi, x \rangle}d\mu(x), ~~~\xi \in \mathbb{R}. $$

It is widely known that the self-similar measure $\mu_{\rho,D}$ can be represented as an infinite convolution of discrete measures
\begin{align}\label{2.1}
	\mu_{\rho,D}=\delta_{\rho D}\ast\delta_{\rho^2 D}\ast\delta_{\rho^3 D}\ast\cdots
\end{align}
where $\delta_{D}=\frac{1}{\#D}\sum_{d\in D}\delta_{d}$, $\delta_{d}$ is the Dirac measure at the point $d\in D$. By direct calculations, the Fourier transform of the self-similar measure $\mu_{\rho,D}$ in ($\ref{2.1}$) is as follows:

\begin{equation}
	\hat{\mu}_{\rho,D}(\xi)=\prod^{\infty}_{n=1}\hat{\delta}_{\rho^n D}(\xi)=\prod^{\infty}_{n=1}m_{D}(\rho^{n}\xi),  \notag
\end{equation}
where $m_D(\cdot)$ is the mask polynomial of the digit set $D$, i.e.,
\begin{align*}
	m_{D}(x)=\frac{1}{\#D}\sum_{d\in D}e^{2\pi i \langle d,x \rangle}, x\in \mathbb{R}.
\end{align*}
We denote $Z(f)=\{x:f(x)=0\}$ as the zero set of function $f$, then it is easy to know

\begin{align}\label{2.2}	Z(\hat{\mu}_{\rho,D})=\bigcup^{\infty}_{n=1}\rho^{-n}Z(m_{D}).
\end{align}

Notice that the properties of spectra are invariant under translation. Without loss of generality, we always assume that $0\in\Lambda$.
For any $\lambda_{1} \neq \lambda_{2}\in\Lambda\subset\mathbb{R}$, the orthogonality condition shows that
$$\langle e^{2\pi i \langle \lambda_{1}, x \rangle},e^{2\pi i \langle \lambda_{2}, x \rangle }\rangle_{L^{2}(\mu_{\rho,D})}=\int _{\mathbb{R}} e^{2\pi i(\lambda_{1}-\lambda_{2})x}d\mu_{\rho,D}=\hat{\mu}_{\rho,D}(\lambda_{1}-\lambda_{2})=0.  $$
Consequently, for a countable set $0\in\Lambda\subset\mathbb{R}$, the family of exponential functions $\mathcal{E}_{\Lambda}$ forms an orthogonal set of $L^{2}(\mu_{\rho,D})$ if and only if
\begin{align}\label{2.3}
	\Lambda \subset (\Lambda-\Lambda)\backslash\{0\}\subset Z(\hat{\mu}_{\rho,D}).
\end{align}

Jorgensen and Pedersen \cite{Jorgensen-Pedersen_1998} provided the following criterion for a countable set $\Lambda$ to be an orthogonal set or a spectrum for the measure $\mu$.

\begin{theorem}[\cite{Jorgensen-Pedersen_1998}] \label{2.15}
	Let $\mu$ be a Borel probability measure with compact support on $\mathbb{R}^n$, and let $Q_{\mu,\Lambda}(\cdot)=\sum_{\lambda \in \Lambda} \vert \hat{\mu}(\cdot + \lambda) \vert ^2$ for a countable set $\Lambda \subset \mathbb{R}^n$. Then
	\begin{enumerate}[\rm(i)]
		\item $\Lambda$ is an orthogonal set of $\mu$ if and only if $Q_{\mu,\Lambda}(x) \le 1$ for all $x\in \mathbb{R}^n$;
		\item $\Lambda$ is a spectrum of $\mu$ if and only if $Q_{\mu,\Lambda}(x) \equiv 1$ for all $x\in \mathbb{R}^n$.
	\end{enumerate}
\end{theorem}

Based on Theorem \ref{2.15}, Dai et al.~\cite{Dai-He-Lau_2014} gave the following criterion for the non-spectrality of a measure.

\begin{theorem}[\cite{Dai-He-Lau_2014}] \label{2.20}
	Let $\mu=\mu_1 \ast \mu_2$ be the convolution of two probability measures $\mu_1$ and $\mu_2$, neither of which is a Dirac measure. Suppose that $\Lambda$ is an orthogonal set of $\mu_1$, then $\Lambda$ is also an orthogonal set of $\mu$, but it cannot be a spectrum of $\mu$.
\end{theorem}

The following consequence follows directly from Theorem \ref{2.20} and provides a convenient criterion for detecting non-spectral self-similar measures.
\begin{remark}\label{2.30}
	{\rm
Let $\mu_{\rho,D}$ be a self-similar measure defined by \eqref{1.1}. Assume that the digit set $D$ admits a decomposition 
$$D=D_{(1)} \oplus D_{(2)} \oplus \cdots \oplus D_{(n)}.$$
 If there exist distinct indices $1 \le j_1 \neq j_2 \le n$ and integers $m_1,m_2 \ge 1$ such that
\begin{equation}
	Z(\hat{\delta}_{\rho^{m_1}D_{(j_1)}}) \subset Z(\hat{\delta}_{\rho^{m_2}D_{(j_2)}}), \notag
\end{equation}
then Theorem \ref{2.20} implies that $\mu_{\rho,D}$ is not a spectral measure.}
\end{remark}

Under a natural structural assumption, An and Wang \cite{An2021} provided a necessary condition for the spectrality of the self-similar measures generated by \eqref{1.1}.
\begin{theorem}[\cite{An2021}] \label{2.4}
	Let $0< \rho <1$ be a real number and $D \subset \mathbb{R}$ be a finite set. Assume that $Z(m_D) \subset \alpha \mathbb{Z}$ for $\alpha \in \mathbb{R} \setminus \{0\}$. If $\mu_{\rho,D}$ is a spectral measure, then $\rho^{-1} \in \mathbb{N}$.
\end{theorem}

Strichartz \cite{Strichartz_2000} introduced the notion of  Hadamard triple (Definition \ref{2.5}), which plays an important role in the study of spectral measures. The following lemma provides a concrete criterion for verifying whether a given triple forms a Hadamard triple.
\begin{lemma}[\cite{Dut-Hau-Lai2019}]\label{lem2-6}
	Let $R\in M_n(\mathbb{Z})$ be an expanding matrix, and let $B,L\subset\mathbb{Z}^n$ be two finite digit sets with the same cardinality. Then $(R,B,L)$ is a Hadamard triple if and only if $m_B(R^{*-1}(l_1-l_2))=0$ for any $l_1\neq l_2\in L$.
\end{lemma}

Let $\mu$ be a Borel probability measure on $\mathbb{R}$. Define the integral periodic zero set of the Fourier transform $\hat{\mu}$ by
$$\mathcal{Z}(\mu)=\{ \xi \in  \mathbb{R}: \hat{\mu}(\xi +k)=0~~ {\rm{for~ all}} ~k \in \mathbb{Z} \}.$$
In general, it is difficult to determine whether the set $\mathcal{Z}(\mu)$ is empty. Consider the infinite convolution
\begin{align}\label{2.9}
	\nu:=\delta_{R_1^{-1}B_1}*\delta_{(R_1R_2)^{-1}B_2}*\cdots*\delta_{(R_1R_2\cdots R_k)^{-1}B_k}*\cdots
\end{align}
generated by a sequence of admissible pairs ${(R_k,B_k)}_{k=1}^{\infty}$. Li et al.~\cite{Li2024} derived a sufficient condition for $\mathcal{Z}(\nu)=\emptyset$ when $(R_k,B_k)$ is selected from a finite collection of admissible pairs.
When $(R_k,B_k)=(R,B)$ for all $k\geq 1$, the convolution in \eqref{2.9} reduces to the self-similar measure
$\mu_{R^{-1},B}$ given by \eqref{2.1}.

\begin{proposition}[\cite{Li2024}] \label{2.10}
	Let the sequence of admissible pairs $\{(R_k,B_k)\}^{\infty}_{k=1}$ be chosen from a finite set of admissible pairs in $\mathbb{R}$, and let $\nu$ be defined by \eqref{2.9}. If
	\begin{equation*}	\gcd\left(\bigcup^{\infty}_{k=n}\left(B_k-B_k\right)\right)=1
	\end{equation*}
	for each $n \ge 1$, then $\mathcal{Z}(\nu)=\emptyset$.
\end{proposition}
Hadamard triples also yield a large class of self-affine spectral measures. In particular,  it has been shown  that if a self-affine measure $\mu$ is generated by an admissible pair $(R, B)$ and  $\mathcal{Z}(\mu)=\emptyset$, then it admits an integer spectrum.
\begin{theorem}[\cite{Dut-Hau-Lai2019}] \label{2.12}
	Let $\mu:=\mu_{R^{-1},B}$ be a self-affine measure generated by \eqref{1.1}, where $(R,B)$ is an admissible pair.
	Then $\mathcal{Z}(\mu)= \emptyset$ if and only if
	$\mu$ has a spectrum in $\mathbb{Z}^n$.
\end{theorem}

The similarity transformation implies that both spectrality and translational tiling are preserved under the action of matrix operators, the proof of which can be found in \cite{Chen-Liu-Zheng2024} and the references therein.
	\begin{theorem}[\cite{Chen-Liu-Zheng2024}] \label{2.7}
		Let $R_1,R_2 \in M_n(\mathbb{R})$ be two expanding matrices, and let $B_1,B_2 \subset \mathbb{R}^n$ be two finite digit sets with the same cardinality. If there exists a matrix $Q \in M_n(\mathbb{R})$ such that $R_2=QR_1Q^{-1}$ and $B_2=QB_1$, then
		\vspace{-0.2cm}
		\begin{enumerate}[\rm(i)]
			\item $\mu_{R_1,B_1}$ is a spectral measure with a spectrum $\Lambda$ if and only if $\mu_{R_2,B_2}$ is a spectral measure with a spectrum $Q^{\ast -1} \Lambda$;\vspace{-0.2cm}
			\item $T(R_1,B_1)$  is a translation tile with a tiling set $\Gamma$ if and only if $T(R_2,B_2)$ is a translation tile with a tiling set $Q \Gamma$.
		\end{enumerate}
\end{theorem}

\section{Spectrality of self-similar measures}\label{sec3}
This section is devoted to studying the spectrality of self-similar measures associated with product-form like digit sets given by \eqref{eq1.3}. We prove Theorem \ref{1.2} by establishing the necessity and sufficiency separately in Theorems \ref{thm3.2} and \ref{thm3.3}.

\medskip


Recall that $\mu_{\rho,D}$ is a self-similar measure generated by $0<\rho<1$ and
\[
D = D_N \oplus m D_L,
\]
where $N,L,m \in \mathbb{N}^*$ with $N,L \ge 2$.
We now give an equivalent description of  $d$ defined in \eqref{2.11}
in terms of the prime factorization of $L$.
Let
\begin{equation}\label{3.5}
	L = L_1^{\alpha_1} L_2^{\alpha_2} \cdots L_\kappa^{\alpha_\kappa}
\end{equation}
be the prime factorization of $L$, where $L_i$ are distinct primes greater than two, and
$\alpha_i \ge 1$ for $1 \le i \le \kappa$. 
Notice that the spectrality of $\mu_{\rho,D}$
(see Proposition~\ref{pro1.1}) imposes a constraint on the parameter $L$, namely $L \mid p$ with $p := \rho^{-1} \in \mathbb{N}^*$. This allows us to also write
\begin{equation}\label{3.6}
	m = L_1^{\tau_1} L_2^{\tau_2} \cdots L_\kappa^{\tau_\kappa} m',
	\qquad
	p = L_1^{l_1} L_2^{l_2} \cdots L_\kappa^{l_\kappa} p',
\end{equation}
where $\tau_i\geq 0$, $l_i\ge \alpha_i$ and
$
\gcd(L_i,m') = \gcd(L_i,p') = 1
~ \text{for all } 1 \le i \le \kappa.
$ 
 Therefore, 
for each $ i $, there exist unique integers $d_i \ge 0$
and $0 \le r_i < l_i$ such that
\begin{equation}\label{3.7}
	\tau_i + \alpha_i - 1 = d_i l_i + r_i .
\end{equation}
Thus $d$ defined in \eqref{2.11} can be equivalently expressed as
\begin{equation}\label{eq3.1.0}
	d = \max_{1 \le i \le \kappa} d_i .
\end{equation}

\vspace{-0.5cm}
\subsection{The necessity.}
\begin{proposition}\label{pro1.1}
Let $\mu_{\rho,D}$ be a self-similar measure defined by \eqref{1.1}, where
$0<\rho<1$ and $D$ is given by \eqref{eq1.3}.
If $\mu_{\rho,D}$ is a spectral measure, then $\rho^{-1}$ is a positive integer that is divisible by both $N$ and $L$.
\end{proposition}

\begin{proof}
It can be checked that $Z(m_D)$ is contained in a lattice set. Hence, by Theorem \ref{2.4}, the hypothesis that $\mu_{\rho,D}$ is a spectral measure implies $\rho^{-1} \in \mathbb{N}^*$. Assume that $\rho^{-1}=p$, then we will show that $N$ and $L$ must be factors of $p$.

First, we prove that $N \mid p$. Suppose for contradiction that this is not the case, and let
$d=\gcd(N,p)$. Since $1\le d<N$, there exist positive integers $p'$ and $N'\geq 2$ such that $p = dp'$, $N = dN'$ and $\gcd(p',N')=1$.
At this time, $D$ can be discomposed into
$$D=D_d \oplus dD_{N'}\oplus mD_L,$$
which implies that the self-similar measure $\mu_{\rho,D}$ can be expressed as
\begin{align*}
\mu_{\rho,D}=&\delta_{\rho D} \ast \delta_{\rho^2 D} \ast \delta_{\rho^3 D} \ast \cdots  \notag  \\
=&(\delta_{\rho D_N}* \delta_{\rho mD_L})* ( \delta_{\rho^2 D_d}  \ast \delta_{\rho^2 d D_{N'}}\ast \delta_{\rho^2 mD_L}) \ast \delta_{\rho^3 D} \cdots.
\end{align*}
By using $\gcd(p',N')=1$, we can obtain
\begin{align*}
	Z(\hat{\delta}_{\rho^2 d D_{N'}})
	&= \frac{p^2}{dN'}\big(\mathbb{Z}\setminus N'\mathbb{Z}\big)  \\
	&= \frac{pp'}{N'}\big(\mathbb{Z}\setminus N'\mathbb{Z}\big)  \\
	&\subset \frac{p}{N'}\big(\mathbb{Z}\setminus N'\mathbb{Z}\big) \\
	&\subset \frac{p}{N}\big(\mathbb{Z}\setminus N\mathbb{Z}\big)
	= Z(\hat{\delta}_{\rho D_N}).
\end{align*}
It follows from Remark $\ref{2.30}$ that $\mu_{\rho,D}$ is not a spectral measure. Hence we have $N \mid p$.

The conclusion $L \mid p$ can be proved by a similar argument, and hence we omit the details for brevity.

\end{proof}

To prove the necessity of Theorem \ref{1.2}, we also need some new notation, defined as follows.  
Let $\sigma = \gcd(m, p^d)$. We can then write $m = \sigma \tilde{m}$ and  $p^d = \sigma \tilde{p}$ for some integers   $\tilde{m}$ and $\tilde{p}$ satisfying $\gcd(\tilde{m}, \tilde{p}) = 1$. Now, define
$$\tilde{d} = \gcd(N, \tilde{m}).$$ Then $N = \tilde{d} N_1$ and $\tilde{m} = \tilde{d} m_1$ for some $N_1, m_1 \in \mathbb{Z}$ with $\gcd(N_1, m_1) = 1$.
Consequently, $N$ and $m$ can be expressed as
\begin{align}\label{3.3}
N = \tilde{d} N_1 \quad \text{and} \quad m = \sigma \tilde{d} m_1.
\end{align}
Without loss of generality, we assume that $d_1=d=\max\{d_i : 1 \leq i \leq \kappa\}$. Notice that $d_1$ depends on $\alpha_1$, $\tau_1$ and $l_1$ by \eqref{3.7}. Thus, by factoring out $L_1$ from the prime factorization of $L$, we write $L = L_1 L'$, where $L' = L_1^{\alpha_1-1} L_2^{\alpha_2} \cdots L_\kappa^{\alpha_\kappa}$. Then the digit set $D$ decomposes as
\begin{equation}\label{D-S-Decompose}
D= D_N \oplus m D_L = \left(D_{\tilde{d}} \oplus \tilde{d} D_{N_1} \right) \oplus m \left( D_{L'} \oplus L' D_{L_1} \right).
\end{equation}
It is clear to see that if  $N_1=1$, then $D_{\tilde{d}} \oplus \tilde{d} D_{N_1}=D_N$. Hence, we shall restrict our attention to  $N_1>1$.
Combining \eqref{D-S-Decompose} with formulas \eqref{2.1} and \eqref{2.2}, we define the following three subsets related to the zero set $Z(\hat{\mu}_{\rho,D})$:
\begin{align}
	R_1 &:= Z(\hat{\delta}_{\rho \tilde{d}D_{N_1}})
	= p \frac{\mathbb{Z} \setminus N_1 \mathbb{Z}}{\tilde{d} N_1},
	\label{eq-R1} \\[0.5em]
	S_1 &:= Z(\hat{\delta}_{\rho^{d+1} mL'D_{L_1}})
=p^{d+1} \frac{\mathbb{Z} \setminus L_1 \mathbb{Z}}{L m}
=p \frac{\tilde{p} (\mathbb{Z} \setminus L_1 \mathbb{Z})}
	{L \tilde{d} m_1},
	\label{eq-S1} \\[0.5em]
	S_2 &:= Z(\widehat{\mu'})
	= \left( \bigcup_{j=1}^{d} p^j
	\frac{\mathbb{Z} \setminus L \mathbb{Z}}{m L} \right)
	\cup \left( p^{d+1}
	\frac{\mathbb{Z} \setminus L' \mathbb{Z}}{L' m} \right),
	\label{eq-S2}
\end{align}
where $\mu':=\delta_{\rho mD_L}*\delta_{\rho^2 mD_L}*\cdots*\delta_{\rho^d mD_L}*\delta_{\rho^{d+1}m D_{L'}}$.

\medskip
With the above preparations, we now prove the following proposition.

\begin{proposition}\label{empty-set}
Let $\mu_{p^{-1},D}$ be a self-similar measure defined by \eqref{1.1}, where $p\in\mathbb{N}^*$ and $D$ is given by \eqref{eq1.3}.  Suppose that $L\mid p$, $N \nmid \frac{m}{\gcd(m, p^d)}$ and $0 \in \Lambda$ is a bi-zero set of $\mu_{p^{-1},D}$. Then
\begin{equation*}
\Lambda\cap (R_1 \setminus (S_1\cup S_2))=\emptyset \quad {\rm{or}} \quad\Lambda\cap (S_1 \setminus (R_1\cup S_2))=\emptyset,
\end{equation*}
where $R_1, S_1$ and $S_2$ are defined in \eqref{eq-R1}, \eqref{eq-S1} and \eqref{eq-S2}, respectively.
\end{proposition}

\begin{proof}
Under the assumption $N \nmid \frac{m}{\gcd(m, p^d)}$, the value $N_1$ defined in \eqref{3.3} is greater than one. 
Suppose, for contradiction, that the conclusion does not hold.
 Then by \eqref{eq-R1} and \eqref{eq-S1}, there exist  $a_1 \in \mathbb{Z} \setminus N_1\mathbb{Z}$  and  $ a_2 \in \mathbb{Z} \setminus L_1\mathbb{Z}$, along with elements
 $ \lambda_1, \lambda_2 \in \Lambda$, such that
\begin{equation}\label{3.10}
 \lambda_1 = p \frac{a_1}{\tilde{d} N_1} \in \Lambda \cap \left( R_1 \setminus (S_1 \cup S_2) \right)
\end{equation}
and
\begin{equation}\label{3.11}
\lambda_2 =p^{d+1} \frac{a_2}{L m}= p \frac{\tilde{p} a_2}{L \tilde{d} m_1} \in \Lambda \cap \left( S_1 \setminus (R_1 \cup S_2) \right).
\end{equation}
Since $ \Lambda$ is an orthogonal set, we have $ \lambda_1 - \lambda_2 \in Z(\hat{\mu}_{p^{-1},D})$ by \eqref{2.3}. Then,
$\lambda_1 - \lambda_2$ must satisfy either
$$\lambda_1-\lambda_2\in \bigcup_{j=1}^{\infty}p^j\frac{\mathbb{Z} \setminus N\mathbb{Z}}{N}\quad {\rm{or}} \quad \lambda_1-\lambda_2\in \bigcup_{j=1}^{\infty}p^j\frac{\mathbb{Z} \setminus L\mathbb{Z}}{Lm}.$$

\textbf{Case 1:} $\lambda_1-\lambda_2 = p^{j_1} \frac{a_3}{N}$ for some integers $j_1\geq 1$ and $a_3 \in \mathbb{Z} \setminus N\mathbb{Z}$.

Since $N = \tilde{d} N_1$, we obtain
$$p \frac{a_1}{\tilde{d}N_1} -p \frac{\tilde{p}a_2}{L\tilde{d}m_1} = p^{j_1} \frac{a_3}{\tilde{d} N_1}.$$
Multiplying both sides by  $N_1Lm_1$  leads to
$$
Lm_1(a_1-p^{j_1-1}a_3)=N_1\tilde{p}a_2.
$$
Hence $L m_1 \mid N_1 \tilde{p} a_2$. Furthermore,  we can obtain $\frac{N_1 \tilde{p} a_2}{L m_1} \notin N_1 \mathbb{Z}.$

Indeed, if  $\frac{N_1 \tilde{p} a_2}{L m_1}=N_1 v$  for some  $v \in \mathbb{Z}$, then
\begin{equation}\label{3.9}
\tilde{p} a_2 = L m_1 v.
\end{equation}
Note that
$\tilde{p}=\frac{p^d}{\gcd(m,p^d)}=:L_1^{\sigma_1}p'',$
where $\sigma_1=dl_1-\min\{\tau_1, dl_1\}\geq 0$  and $p''\in\mathbb{Z}$ satisfies $\gcd(p'',L_1)=1$. Combining $d_1=d$ and \eqref{3.7}, one may obtain that
$$\tau_1 + \alpha_1 - 1 = d l_1 + r_1 \quad (0 \leq r_1 < l_1).$$
This shows that $0\leq\sigma_1\leq \alpha_1-1$. By \eqref{3.9} and $L = L_1^{\alpha_1} L_2^{\alpha_2} \cdots L_\kappa^{\alpha_\kappa}$, we have
\begin{align}\label{eq3.13}
	p''a_2=L_1^{\alpha_1-\sigma_1}L_2^{\alpha_2} \cdots L_\kappa^{\alpha_\kappa}m_1v.
\end{align}
Since  $\gcd(p'' a_2, L_1)=1$ and $\alpha_1-\sigma_1\geq 1$, the equation \eqref{eq3.13} leads to a contradiction.
 Therefore, we get $\frac{N_1 \tilde{p} a_2}{L m_1} \notin N_1 \mathbb{Z}$ and 
$$
\lambda_2 = p \frac{\tilde{p} a_2}{L \tilde{d} m_1} = p \cdot \frac{1}{\tilde{d} N_1} \cdot \frac{N_1 \tilde{p} a_2}{L m_1} \in p \frac{\mathbb{Z} \setminus N_1 \mathbb{Z}}{\tilde{d} N_1} = R_1,
$$
which contradicts the condition $\lambda_2 \in \Lambda \cap \left( S_1 \setminus (R_1 \cup S_2) \right)$.

\textbf{Case 2:} $\lambda_1-\lambda_2 = p^{j_2} \frac{a_4}{Lm}$ for some integers $j_2\geq 1$ and $a_4 \in \mathbb{Z} \setminus L\mathbb{Z}$.

Substituting \eqref{3.10} and \eqref{3.11} into the expression for $\lambda_1-\lambda_2$ gives
\begin{align}\label{3.4}
\lambda_1-\lambda_2 =p \frac{a_1}{\tilde{d} N_1}- p^{d+1} \frac{a_2}{L m}=p^{j_2} \frac{a_4}{Lm}.
\end{align}
We first show that  $j_2 \geq d + 1$. Suppose, to the contrary, that  $j_2 < d + 1$.
From \eqref{3.4}, $a_4\in\mathbb{Z} \setminus L \mathbb{Z}$, and $L \mid p$, we can obtain that
$$\lambda_1=p^{j_2} \frac{a_4+p^{d + 1 - j_2} a_2}{L m}\in p^{j_2}\frac{\mathbb{Z}\setminus L\mathbb{Z}}{Lm}\subset S_2.$$
This contradicts the assumption $\lambda_1\in\Lambda\cap\left( R_1\setminus(S_1\cup S_2)\right)$. Hence  $j_2 \geq d+1$.

From \eqref{3.4}, along with  $m=\sigma \tilde{d}m_1$ and $p^d=\sigma \tilde{p}$, we further obtain
$$L m_1 a_1 = N_1 \tilde{p}\left( a_2 + p^{j_2 - d - 1} a_4 \right).$$
Indeed, this implies
\begin{align}\label{eq 3.4}
a_2 + p^{j_2 - d-1}  a_4=\frac{Lm_1a_1}{N_1\tilde{p}} \in\mathbb{Z}\setminus L\mathbb{Z}.
\end{align}
Otherwise, we would have $\frac{L m_1 a_1}{N_1 \tilde{p}} = L z_0$
for some $z_0 \in \mathbb{Z}$, and hence
$
m_1 a_1 = N_1 \tilde{p} z_0 \in N_1 \mathbb{Z},
$
which contradicts the assumptions that $\gcd(m_1, N_1)=1$ and
$a_1 \in \mathbb{Z} \setminus N_1 \mathbb{Z}$.
 Thus, it follows from \eqref{eq 3.4} that
 \begin{align*}
	\lambda_1
	&= p\frac{a_1}{\tilde{d} N_1}= p\frac{\tilde{p}\big(a_2 + p^{\,j_2-d-1} a_4\big)}{\tilde{d} m_1 L} \in p^{d+1}\frac{\mathbb{Z}\setminus L\mathbb{Z}}{mL}
	\subset S_1 \cup S_2,
\end{align*}
which again contradicts the condition $\lambda_1 \in \Lambda \cap \left( R_1 \setminus (S_1 \cup S_2) \right)$.
Therefore, both cases lead to a contradiction and we complete the proof.

\end{proof}

Now we continue proving the necessity of Theorem \ref{1.2}.

\begin{theorem}\label{thm3.2}
Let $\mu_{\rho,D}$ be a self-similar measure defined by \eqref{1.1}, where $0<\rho<1$ and $D$ is given by \eqref{eq1.3}. If $\mu_{\rho,D}$ is a spectral measure, then $\rho^{-1}=p\in\mathbb{N}^*$ with $N\mid p$, $L\mid p$ and $N\mid\frac{m}{\gcd(m,p^d)}$, where $d \in \mathbb{N}$ is defined as \eqref{eq3.1.0}.
\end{theorem}

\begin{proof}
By Proposition \ref{pro1.1}, it follows that $\rho^{-1}=p\in\mathbb{N}^*$ with $N\mid p$ and $L\mid p$. It therefore remains to show that $N\mid\frac{m}{\gcd(m,p^d)}$. We argue by contradiction. Assume, to the contrary, that $N\nmid\frac{m}{\gcd(m,p^d)}$.  Let $0\in\Lambda$ be a spectrum of $\mu_{\rho,D}$. A contradiction will be obtained by constructing a nonzero function
$F \in L^2(\mu_{\rho,D})$ satisfying
$\langle F,e_{\lambda} \rangle_{L^2(\mu_{\rho,D})}=0$ for all $e_{\lambda}\in \mathcal{E}_\Lambda$. 

\medskip
Recall that $R_1, S_1$ and $S_2$ are defined by \eqref{eq-R1}, \eqref{eq-S1} and \eqref{eq-S2} respectively.
Then by Proposition  \ref{empty-set}, we have
$$
\Lambda\cap (R_1 \setminus (S_1\cup S_2))=\emptyset \quad {\rm{or}} \quad\Lambda\cap (S_1 \setminus (R_1\cup S_2))=\emptyset.
$$
In view of the decomposition of $D$ in \eqref{D-S-Decompose}, we treat the two possibilities separately.

\medskip
\textbf{Case 1: $\Lambda\cap (R_1 \setminus (S_1\cup S_2))=\emptyset$.}

Note that the self-similar measure $\mu_{\rho,D}$ is supported on $T(p^{-1},D)$, which satisfies the equation
$$T(p^{-1},D)=p^{-1}D +\sum_{k \ge 2}p^{-k}D =: p^{-1}\tilde{d}D_{N_1}+ \tilde{T}_1$$
with $\tilde{T}_1:= p^{-1}D_{\tilde{d}}+p^{-1}mD_{L}+\sum_{k \ge 2}p^{-k}D$ from the decomposition $D=\tilde{d}D_{N_1}\oplus D_{\tilde{d}}\oplus mD_{L}$. The measure $\mu_{\rho,D}$  also admits the convolution representation

\begin{equation*}
	\mu_{\rho,D} = \delta_{p^{-1}\tilde{d}D_{N_1}} \ast \mu_1 = \frac{1}{N_1} \sum_{a \in p^{-1}\tilde{d}D_{N_1}} \delta_a \ast \mu_1,
\end{equation*}
where $\mu_1:=\delta_{p^{-1}D_{\tilde{d}}} \ast \delta_{p^{-1} mD_{L}} \ast (\ast^{\infty}_{j=2} \delta_{p^{-j}D}).$
For each $a \in p^{-1}\tilde{d}D_{N_1}$, we define the shifted probability measure $\nu_a = \delta_a \ast \mu_1$. By the convolution relation of $\mu_{\rho,D}$, we have the sum representation:
\begin{equation*}
	\mu_{\rho,D} = \frac{1}{N_1} \sum_{a \in p^{-1}\tilde{d}D_{N_1}} \nu_a.
\end{equation*}
Since all $\nu_a$ are positive measures, for any Borel set $E \subset \mathbb{R}$ and any $a\in p^{-1}\tilde{d}D_{N_1}$, it follows that $\mu_{\rho,D}(E) \ge \frac{1}{N_1} \nu_a(E)$. This implies that each measure $\nu_a$ is absolutely continuous with respect to the invariant measure $\mu_{\rho,D}$.
By the Radon-Nikodym theorem, for each $a\in p^{-1}\tilde{d}D_{N_1}$, there exists a density function $h_a = \frac{d\nu_a}{d\mu_{\rho,D}} \in L^1(\mu_{\rho,D})$ such that $0 \le h_a(x) \le N_1$ holds for $\mu_{\rho,D}$-almost everywhere on $ T(p^{-1},D)$.

The assumption $N\nmid\frac{m}{\gcd(m,p^d)}$ implies that $N_1>1$, hence
 $0$ and $p^{-1}\tilde{d}$ are both belong to $p^{-1}\tilde{d}D_{N_1}$. Consequently, their corresponding Radon-Nikodym derivatives $h_0(x)$ and $h_{p^{-1}\tilde{d}}(x)$ must exist.
Now, we construct a function $F(x)$ purely in a measure-theoretic sense as:
\begin{equation}\label{3.13}
	F(x) = \frac{1}{N_1} \Big( h_0(x) - h_{p^{-1}\tilde{d}}(x) \Big).
\end{equation}
Because $0 \le h_a(x) \le N_1$ almost everywhere for all $a\in p^{-1}\tilde{d}D_{N_1}$, we  have $|F(x)| \le 1$ for $\mu_{\rho,D}$-almost every $x$. Therefore, $\int |F(x)|^2 d\mu_{\rho,D}(x) \le 1 < \infty$, which  guarantees that $F \in L^2(\mu_{\rho,D})$. 
Next, we compute the inner product of $F$ with the exponential function $e_\lambda(x) = e^{2\pi i \lambda x},\lambda\in \Lambda$: 
\begin{equation*}
	\begin{aligned}
		\langle F, e_\lambda \rangle_{L^2(\mu_{\rho,D})} &= \int_{\mathbb{R}} F(x) e^{-2\pi i \lambda x} d\mu_{\rho,D}(x) \\
		&= \frac{1}{N_1} \int_{\mathbb{R}} e^{-2\pi i \lambda x} h_0(x) d\mu_{\rho,D}(x) - \frac{1}{N_1} \int_{\mathbb{R}} e^{-2\pi i \lambda x} h_{p^{-1}\tilde{d}}(x) d\mu_{\rho,D}(x) \\
		&= \frac{1}{N_1} \int_{\mathbb{R}} e^{-2\pi i \lambda x} d\nu_0(x) - \frac{1}{N_1} \int_{\mathbb{R}} e^{-2\pi i \lambda x} d\nu_{p^{-1}\tilde{d}}(x).
	\end{aligned}
\end{equation*}
The last equality uses  the definition of the Radon-Nikodym derivative, hence we can transition the integration from  $\mu_{\rho,D}$ to the shifted measures $\nu_a$.
Note that $\nu_0 = \mu_1$ and $\nu_{p^{-1}\tilde{d}} = \delta_{p^{-1}\tilde{d}} \ast \mu_1$.  Applying the translation property of the Fourier transform, we obtain
\begin{equation*}
	\begin{aligned}
		\langle F, e_\lambda \rangle_{L^2(\mu_{\rho,D})} &= \frac{1}{N_1} \widehat{\mu}_1(\lambda) - \frac{1}{N_1} e^{-2\pi i p^{-1}\tilde{d} \lambda} \widehat{\mu}_1(\lambda) \\
		&= \frac{1}{N_1} \left( 1 - e^{-2\pi i \frac{\tilde{d}}{p} \lambda} \right) \widehat{\mu}_1(\lambda).
	\end{aligned}
\end{equation*}
Recall that $\mu_1$ is defined via the infinite convolution:
\begin{equation*}
	\mu_1 = \delta_{p^{-1}D_{\tilde{d}}} \ast \delta_{p^{-1}mD_L} \ast \left( \mathop{\ast}_{k=2}^\infty \delta_{p^{-k}D} \right).
\end{equation*}
Taking the Fourier transform of both sides yields the infinite product expansion:
\begin{equation*}
	\widehat{\mu}_1(\lambda) = \widehat{\delta}_{D_{\tilde{d}}}\left(\frac{\lambda}{p}\right) \widehat{\delta}_{D_L}\left(\frac{m\lambda}{p}\right) \prod_{k = 2}^{\infty} \widehat{\delta}_D(p^{-k}\lambda).
\end{equation*}
Substituting this expansion into our expression for the inner product, we arrive at the formula:
\begin{equation*}
	\langle F, e_\lambda \rangle_{L^2(\mu_{\rho,D})} = \frac{1}{N_1} \left( 1 - e^{-2\pi i \frac{\tilde{d}}{p}\lambda} \right) \widehat{\delta}_{D_{\tilde{d}}}\left(\frac{\lambda}{p}\right) \widehat{\delta}_{D_L}\left(\frac{m\lambda}{p}\right) \prod_{k = 2}^{\infty} \widehat{\delta}_D(p^{-k}\lambda).
\end{equation*}
For any nonzero $\lambda \in \Lambda$, the condition $\Lambda\cap (R_1 \setminus (S_1\cup S_2))=\emptyset$ implies
$$\hat{\delta}_{D_{\tilde{d}}}(\frac{\lambda}{p}) \cdot\hat{\delta}_{D_L}(\frac{m\lambda}{p}) \cdot \prod_{k \ge 2}\hat{\delta}_D(p^{-k}\lambda)=0.$$
If $\lambda=0$, then $1- e^{-2 \pi i \frac{\tilde{d}}{p} \lambda}=0$.
Hence we have $\langle F,e_{\lambda} \rangle_{L^2(\mu_{\rho,D})}=0$ for each $e_{\lambda}\in \mathcal{E}_\Lambda$.

Finally, we show that the function $F \in L^2(\mu_{\rho,D})$ above constructed is non-zero. We argue by contradiction. Suppose that $F(x) = 0$ for $\mu_{\rho,D}$-almost every $x$. The definition of $F(x)$ in \eqref{3.13} forces 
\begin{equation*}
	h_0(x) = h_{p^{-1}\tilde{d}}(x), \quad \mu_{\rho,D}\text{-a.e.}
\end{equation*}
Consequently, for any Borel set $A \subset T(p^{-1},D)$, their integrals with respect to the measure $\mu_{\rho,D}$ must coincide:
\begin{equation*}
	\nu_0(A) = \int_A h_0(x) d\mu_{\rho,D}(x) = \int_A h_{p^{-1}\tilde{d}}(x) d\mu_{\rho,D}(x) = \nu_{p^{-1}\tilde{d}}(A).
\end{equation*}
This implies that $\nu_0 = \nu_{p^{-1}\tilde{d}}$. Recalling the definitions $\nu_0 = \mu_1$ and $\nu_{p^{-1}\tilde{d}} = \delta_{p^{-1}\tilde{d}} \ast \mu_1$, we obtain the following measure convolution equation:
\begin{equation*}
	\mu_1 = \delta_{p^{-1}\tilde{d}} \ast \mu_1.
\end{equation*}
Taking the Fourier transform on both sides of this equation, we get
\begin{equation*}
	\widehat{\mu}_1(\lambda) = e^{-2\pi i p^{-1}\tilde{d} \lambda} \widehat{\mu}_1(\lambda), \quad \forall \lambda \in \Lambda,
\end{equation*}
which can be rewritten as
\begin{equation} \label{eq:contradiction_fourier}
	\widehat{\mu}_1(\lambda) \left( 1 - e^{-2\pi i p^{-1}\tilde{d} \lambda} \right) = 0, \quad \forall \lambda \in \Lambda.
\end{equation}

Since the Fourier transform $\widehat{\mu}_1(\lambda)$ is uniformly continuous on $\mathbb{R}$ and satisfies $\widehat{\mu}_1(0) > 0$ at the origin, continuity implies that there exists a sufficiently small neighborhood around the origin, say $(-\epsilon, \epsilon)$ with $\epsilon > 0$, such that $\widehat{\mu}_1(\lambda) \neq 0$ for all $\lambda \in (-\epsilon, \epsilon)$.
Moreover, because $p^{-1}\tilde{d} \neq 0$, we can choose a nonzero $\lambda_0 \in (-\epsilon, \epsilon)$ small enough such that $p^{-1}\tilde{d} \lambda_0 \notin \mathbb{Z}$. For this specific $\lambda_0$, the exponential term satisfies
\begin{equation*}
	1 - e^{-2\pi i p^{-1}\tilde{d} \lambda_0} \neq 0.
\end{equation*}
Multiplying these two non-zero terms yields 
\begin{equation*}
	\widehat{\mu}_1(\lambda_0) \left( 1 - e^{-2\pi i p^{-1}\tilde{d} \lambda_0} \right) \neq 0.
\end{equation*}
This  contradicts with  \eqref{eq:contradiction_fourier}, which requires the product to be identically zero for all $\lambda \in \Lambda$.
Therefore, we conclude that $F(x)$ cannot be zero $\mu_{\rho,D}$-almost everywhere.

\medskip

\textbf{Case 2: $\Lambda\cap (S_1 \setminus (R_1\cup S_2))=\emptyset$.}

Similar to the discussion in the Case 1, the self-similar set $T(p^{-1},D)$ can be written as
$$T(p^{-1},D)=p^{-(d+1)}D+\sum_{k \ge 1,k \neq d+1}p^{-k}D=:p^{-(d+1)}mL'D_{L_1}+\tilde{T}_2,$$
where $\tilde{T}_2:=p^{-(d+1)}D_{N}+p^{-(d+1)}mD_{L'} +\sum_{k \ge 1,k \neq d+1}p^{-k}D $.
In this case,  $\mu_{\rho,D}$ can be expressed as 
$$\mu_{\rho,D}=\frac{1}{L_1} \sum_{\ell \in p^{-(d+1)}mL'D_{L_1}}\delta_{\ell} \ast \mu_2,$$
where $\mu_2:=\delta_{p^{-(d+1)}D_{N}} \ast \delta_{p^{-(d+1)}mD_{L'}} \ast (\ast^{\infty}_{j=1,j \neq d+1} \delta_{p^{-j}D})$ is supported on $\tilde{T}_2$.
 For each $\ell \in p^{-(d+1)}mL'D_{L_1}$, we define  $\varrho_{\ell} = \delta_{\ell} \ast \mu_2$. Using a similar argument, one can obtain that for each such $\ell$, there exists a density function $s_{\ell} = \frac{d\varrho_{\ell}}{d\mu_{\rho,D}} \in L^1(\mu_{\rho,D})$ such that $0 \le s_{\ell} \le L_1$ holds for $\mu_{\rho,D}$-almost every $x \in T(p^{-1},D)$. 
Since $L_1>1$, we can select the two Radon-Nikodym derivatives $s_{0}(x)$ and $s_{p^{-(d+1)}mL'}(x)$ and define the function $G(x)$ as
\begin{equation*}
	G(x) = \frac{1}{L_1} \Big( s_{0}(x) - s_{p^{-(d+1)}mL'}(x) \Big) \in L^2(\mu_{\rho,D}).
\end{equation*}
In the same manner, one can cheak that 
$\langle G,e_{\lambda} \rangle_{L^2(\mu_{\rho,D})}=0$ for any $e_{\lambda}\in \mathcal{E}_\Lambda$,
 and that $G$ is nonzero in $L^2(\mu_{\rho,D})$. 
 Hence we obtain a contradiction, which implies that $N\mid\frac{m}{\gcd(m,p^{d})}$.
\end{proof}

\subsection{The sufficiency.}

In this subsection, we prove the sufficiency of Theorem \ref{1.2}.
Under the assumption that $\rho^{-1}=p\in\mathbb{N}^*$ with $L$ dividing $p$, the numbers $L$, $m$, and $p$ admit prime factorizations as given in equations \eqref{3.5} and \eqref{3.6},  which satisfy  \eqref{3.7} and \eqref{eq3.1.0}. 
Together with the condition that $N$ divides $\frac{m}{\gcd(m,p^d)}$, this implies the following crucial observation.

\begin{proposition}\label{pro3.4}
Suppose that $L\mid p$,  $N\mid \frac{m}{\gcd(m,p^d)}$ and \eqref{3.5}-\eqref{eq3.1.0} hold. If $L_{i_0}\mid N$ for some $1\leq i_0\leq \kappa$, then $d_{i_0}=d=\max\{d_i : 1 \leq i \leq \kappa\}$.
\end{proposition}

\begin{proof}
Since $L\mid p$, it is clear that $l_i\geq\alpha_{i}\geq 1$ for all $1 \leq i \leq \kappa$. Suppose that there exists some $1\leq i_0\leq \kappa$ such that $L_{i_0}\mid N$ and $d_{i_0}<d$.
It follows from \eqref{3.7} and $0\leq r_{i_0}<l_{i_0}$ that
$$
\tau_{i_0}\leq \tau_{i_0}+\alpha_{i_0}-1=d_{i_0}l_{i_0}+r_{i_0}<dl_{i_0},
$$
		which implies that
		$$
		\gcd\left(L_{{i_0}},\frac{m}{\gcd\left(m,p^{d}\right)}\right)=1.
		$$
		This contradicts the hypotheses that $N\mid \frac{m}{\gcd(m,p^d)}$ and $L_{i_0}\mid N$. Hence, the proposition holds.
\end{proof}

\begin{theorem}\label{thm3.3}
	Let $\mu_{\rho,D}$ be a self-similar measure defined by \eqref{1.1}, where $0<\rho<1$ and $D$ is given by \eqref{eq1.3}.  Suppose that $\rho^{-1}=p\in\mathbb{N}^*$ with $N,L\mid p$ and $N\mid \frac{m}{\gcd(m,p^d)}$, where $d \in \mathbb{N}$ is defined as \eqref{eq3.1.0}. Then $\mu_{\rho,D}$ is a spectral measure.
\end{theorem}

\begin{proof}
Recall the factorizations of $L$, $m$ and $p$ in \eqref{3.5} and \eqref{3.6},  which satisfy  \eqref{3.7} and \eqref{eq3.1.0}. 
Consider the factorization $N = L_{1}^{s_{1}}L_{2}^{s_{2}}\cdots L_{\kappa}^{s_{\kappa}}N'$, where $s_i\geq 0$ and $\gcd(L_i, N')=1$ for $1\leq i\leq \kappa$. Without loss of generality, after relabeling, we may arrange that $s_i = 0$ for $1 \leq i \leq b$ and $s_i \geq 1$ for $b+1 \leq i \leq \kappa$. In other words,
\begin{equation}\label{3.12}
N= L_{b+1}^{s_{b+1}}\cdots L_{\kappa}^{s_{\kappa}}N'.
\end{equation}
Write $\tilde{r}_i:=\tau_i-dl_i$ for $1\leq i\leq \kappa$. For $1\leq i\leq b$, by $d_i\leq d$ and \eqref{3.7}, we have
	\begin{equation}\label{eq3.9.0}
		\tilde{r}_i\leq r_i-\alpha_i+1.
	\end{equation}
For $b+1\leq i\leq\kappa$, it follows from  Proposition \ref{pro3.4} and \eqref{3.7} that
	$$ \tau_i+\alpha_i-1=d_il_i+r_i=dl_i+\tilde{r}_i+\alpha_i-1.
	$$
	Combining with $N\mid\frac{m}{\gcd(m,p^d)}$, one may obtain that
	\begin{equation}\label{eq3.8.0}
		s_i\leq \tilde{r}_i\leq \tilde{r}_i+\alpha_i-1= r_i<l_i, \quad b+1\leq i\leq \kappa.
	\end{equation}

Let $\frac{m}{p^d}=\frac{\tilde{m}}{\tilde{p}}$ with $\gcd(\tilde{m}, \tilde{p})=1$. Then, \eqref{eq3.9.0} and \eqref{eq3.8.0}, together with $N\mid\frac{m}{\gcd(m,p^d)}$, imply that there exists some integer $m_0$ with $\gcd(m_0,L)=1$ such that
$$\tilde{m}=L_1^{\gamma_1}\cdots L_b^{\gamma_b}  L_{b+1}^{\tilde{r}_{b+1}}\cdots L_{\kappa}^{\tilde{r}_\kappa}N'm_0=:\bar{L}N'm_0,$$
where $\gamma_i\in\{0,\tilde{r}_i\}$ for $1\leq i\leq b$. Here,
	$\bar{L}:=L_1^{\gamma_1}\cdots L_b^{\gamma_b} L_{b+1}^{\tilde{r}_{b+1}}\cdots L_{\kappa}^{\tilde{r}_\kappa}$
	satisfies $\bar{L}L\mid p$ since $\gamma_i+\alpha_i\leq l_i$ for $1\leq i\leq b$ and $\tilde{r}_{i}+\alpha_i\leq l_i$ for $b+1\leq i\leq \kappa$. Moreover, $N\mid\frac{m}{\gcd(m,p^d)}$ also implies that $\gcd(\tilde{p},N)=1$.

	By Theorem \ref{2.7}, it suffices to show that $\mu_{p^{-1}, \tilde{p}D}$ is a spectral measure. In the following, we construct a spectrum for $\mu_{p^{-1}, \tilde{p}D}$.  In the case $d\geq 1$, observe that $\mu_{p^{-1}, \tilde{p}D}$ can be  decomposed as
	\begin{align*}
		\mu_{p^{-1}, \tilde{p}D}
		&= \mu_{p^{-1}, m\tilde{p}D_L} * \mu_{p^{-1}, \tilde{p}D_N} \\
		&= \Bigl( \delta_{p^{-1} m\tilde{p}D_L} * \cdots * \delta_{p^{-d} m\tilde{p}D_L} \Bigr)
		* \Bigl( *_{j \ge 1} \delta_{p^{-j}(p^{-d}m\tilde{p}D_L \oplus \tilde{p}D_N)} \Bigr) \\
		&=: \delta_{\mathcal{D}_1} * \mu_{p^{-1},\mathcal{D}_2},
	\end{align*}
	where $\mathcal {D}_1=p^{-1}m\tilde{p}D_L\oplus \cdots\oplus p^{-d}m\tilde{p}D_L$ and $\mathcal {D}_2=p^{-d}m\tilde{p}D_{L}\oplus\tilde{p}D_N$.
The spectra of $\delta_{\mathcal {D}_1}$ and $\mu_{p^{-1},\mathcal {D}_2}$ are constructed separately below.
	
	Let
	$$
	\Lambda_1=\left\{\sum^{d-1}_{k=0} \frac{\eta_k}{\tilde{m} p^k L} \mid \eta_k \in \left\{0,1,\cdots, L-1\right\} \right\},
	$$
	then a direct verification shows that $(\Lambda_1-\Lambda_1)\setminus\{0\}\subset Z(\hat{\delta}_{\mathcal {D}_1})$ and $\#\Lambda_1=L^d$. Hence $\Lambda_1$ is a spectrum of  $\delta_{\mathcal {D}_1}$.
	
	We then verify that the measure $\mu_{p^{-1},\mathcal {D}_2}$ admits an integer spectrum. Note that $$\mathcal{D}_2=\tilde{m}D_{L}\oplus\tilde{p}D_N=\bar{L}N'm_0D_{L}\oplus\tilde{p}D_N$$
 with $\gcd(m_0,L)=\gcd(\tilde{p},N)=1$. Define
	$$
	\mathcal {L}:=\frac{1}{\bar{L}N'}\left\{0,\frac{1}{L},\cdots, \frac{L-1}{L}\right\}\oplus\left\{0,\frac{1}{N},\cdots,\frac{N-1}{N}\right\}.
	$$
	A simple calculation yields that $(\mathcal {L}-\mathcal {L})\setminus\{0\}\subset Z(m_{\mathcal {D}_2})$   by using \eqref{eq3.8.0}. Moreover, it follows from $N\mid p$, $\bar{L}L\mid p$ and $\gcd(N',L)=1$ that $p\mathcal {L}\subset \mathbb{Z}$. Thus, $(p, \mathcal {D}_2)$ is an admissible pair by Lemma \ref{lem2-6}, and consequently $\mu_{p^{-1},\mathcal {D}_2}$ is a spectral measure \cite{Dut-Hau-Lai2019}. Since
	$\gcd(\tilde{p},\tilde{m})=1$,
	we have $\gcd(\mathcal {D}_2 - \mathcal {D}_2)=1$, which shows that $\mathcal{Z}(\mu_{p^{-1},\mathcal {D}_2})=\emptyset$ by Proposition \ref{2.10}. Therefore, Theorem \ref{2.12} guarantees that
	$\mu_{p^{-1},\mathcal {D}_2}$ is a spectral measure possessing a spectrum $\Lambda_2\subset \mathbb{Z}$.
	
	Notice that $\hat{\delta}_{\mathcal {D}_1}$ is integer-periodic since
	$\mathcal {D}_1=\tilde{m}(D_{L}\oplus pD_{L}\oplus\cdots\oplus p^{d-1} D_{L})$ is an integer digit set. Now we set
	$$
	\Lambda=\Lambda_1 \oplus\Lambda_2.
	$$
	Recall that $\mu_{p^{-1}, \tilde{p}D}=\delta_{\mathcal {D}_1} \ast \mu_{p^{-1},\mathcal {D}_2}$. Then it follows that for any $t\in\mathbb{R}$,
	$$
	Q_{\Lambda}(t)=\sum_{\eta \in \Lambda} \vert \hat{\mu}_{p^{-1}, \tilde{p}D}(t+\eta) \vert ^2 = \sum_{\eta \in \Lambda} \vert \hat{\delta}_{\mathcal {D}_1}(t+\eta) \vert ^2 \cdot \vert \hat{\mu}_{p^{-1},\mathcal {D}_2}(t+\eta) \vert ^2.
	$$
	Combining $\Lambda_2\subset\mathbb{Z}$, the integral periodicity of $\hat{\delta}_{\mathcal {D}_1}$ and  Theorem \ref{2.15}, one may obtain that
	\begin{align*}
		Q_{\Lambda}(t)=& \sum_{\gamma \in \Lambda_1}\sum_{\lambda \in \Lambda_2} \vert \hat{\delta}_{\mathcal {D}_1}(t+\gamma+\lambda) \vert ^2 \cdot \vert \hat{\mu}_{p^{-1},\mathcal {D}_2}(t+\gamma+\lambda) \vert ^2  \notag \\
		=& \sum_{\gamma \in \Lambda_1} \vert \hat{\delta}_{\mathcal {D}_1}(t+\gamma) \vert ^2 \cdot \sum_{\lambda \in \Lambda_2} \vert \hat{\mu}_{p^{-1},\mathcal {D}_2}(t+\gamma+\lambda) \vert ^2  \notag \\
		=& \sum_{\gamma \in \Lambda_1} \vert \hat{\delta}_{\mathcal {D}_1}(t+\gamma) \vert ^2 \equiv 1
	\end{align*}
	holds for any $t \in \mathbb{R}$. Applying Theorem \ref{2.15} again, we conclude that
	$\Lambda$ is a spectrum of $\mu_{p^{-1}, \tilde{p}D}$.
In the remaining case $d=0$, $\Lambda_2$ is a spectrum of $\mu_{p^{-1}, \tilde{p}D}$.
The proof is completed.
	
\end{proof}

\begin{remark}\label{rem3.1}
	{\rm 
In the end of this section, we remark that the conditions $N\mid p$, $L\mid p$ and $N\mid \frac{m}{\gcd(m,p^d)}$ imply $NL\mid p$. In fact, using the factorizations of $L,m,p$ and $N$ in \eqref{3.5}-\eqref{eq3.1.0} and \eqref{3.12}, the conditions $N\mid p$ and $L\mid p$ tell us that $L_1^{\alpha_1} \cdots L_b^{\alpha_b}N'\mid p$. 
Moreover, \eqref{eq3.8.0} yields $s_i + \alpha_i \leq l_i$  for all  $b+1 \leq i \leq \kappa$, which ensures $L_{b+1}^{\alpha_{b+1}+s_{b+1}} \cdots L_{\kappa}^{\alpha_{\kappa}+s_{\kappa}} \mid p$. Hence we have $NL \mid p$.
}
\end{remark}

\section{Spectrality of self-similar tiles}\label{sec4}
In Section~\ref{sec3}, we have obtained a complete characterization of the spectrality for self-similar measures $\mu_{\rho,D}$ associated with product-form digit sets $D$ defined by \eqref{eq1.3}. By Remark \ref{rem3.1}, we know that $p=\rho^{-1}\ge \#D$ if $\mu_{\rho,D}$ is a spectral measure. 
In this section, we turn to the tiling properties of the associated self-similar sets $T(\rho, D)$, with the focus on the case 
$$\rho^{-1} = \#D = NL$$
 and on digit sets of the form $D = D_N \oplus m D_L$, where $N, m, L \in \mathbb{N}^*$ with $N, L \ge 2$.  For our study, we recall a fundamental tiling criterion for self-similar tiles. This characterization was first established in one dimension by Kenyon~\cite{Ken1992} and later extended to higher dimensions by Lagarias and Wang~\cite{Lag1997}.

\begin{theorem}[\cite{Lag1997}] \label{2.8}
	For an expanding matrix $R \in M_n(\mathbb{Z})$ and a finite digit set $B \subset \mathbb{Z}^n$ with $\vert \det(R) \vert = \# B$, the following three statements are equivalent:
	\begin{enumerate}[\rm(i)]
		\item $T(R,B)$ has positive Lebesgue measure;
		\item For each $k \ge 1$, the set $\sum^{k-1}_{i=0}R^{i}B$ contains $(\#B)^k$ distinct elements;
		\item For each $\nu \in \mathbb{Z}^n \backslash \{\bf{0}\}$, there exists an integer $k\ge 1$ such that $m_B((R^{\ast})^{-k}\nu)=0.$
	\end{enumerate}	
\end{theorem}
For convenience, we summarize the decomposition process for $L$, $m$, $N$ and $p=\rho^{-1}$ below.
Let $L = L_1^{\alpha_1} L_2^{\alpha_2} \cdots L_\kappa^{\alpha_\kappa}$ with $\alpha_i\geq 1$ for $1\leq i\leq\kappa$ be the prime factorization.  Write
\begin{equation}\label{4.6}
N=L_{b+1}^{s_{b+1}}\cdots L_{\kappa}^{s_{\kappa}}N'\quad {\rm{and}} \quad m = L_1^{\tau_1} L_2^{\tau_2} \cdots L_\kappa^{\tau_\kappa} m'
\end{equation}
for some $0\leq b<\kappa$, where $s_i\geq 1$, $\tau_i\geq0 $ and $\gcd(L_i,N')=\gcd(L_i,m')=1$ for each $i$. The condition $p=LN$ yields the factorization
\begin{equation}\label{4.7}
p = L_1^{l_1} L_2^{l_2} \cdots L_\kappa^{l_\kappa} N'
\end{equation}
with exponents satisfying
\begin{equation}\label{4.8}
l_i=\alpha_i~~~~(1\leq i\leq b) \quad {\rm{and}} \quad l_i=\alpha_i+s_i~~~~(b+1\leq i\leq \kappa).
\end{equation}
For the sake of clarity, we separate the proof of Theorem \ref{1.3} into two propositions (\ref{prop6.2} and \ref{cor6.1}), which collectively establish the theorem.

\begin{proposition}\label{prop6.2}
Let $T(\rho,D)$ be a self-similar set defined by \eqref{1.1}, where $\rho^{-1}=NL$ and $D$ is given by \eqref{eq1.3}.
	If $T(\rho,D)$ is a translation tile, then $N\mid\frac{m}{\gcd(m,(NL)^d)}$ with $d$ given by \eqref{eq3.1.0} and $T(\rho,D)$ is a spectral set.
\end{proposition}
\begin{proof}
Suppose that $T(\rho,D)$ is  a translation tile. By Theorem \ref{2.8}, we have
\begin{align}\label{4.4}
	\mathbb{Z} \setminus \{0\} \subset Z(\hat{\mu}_{\rho,D})=\bigcup^{\infty}_{k=1}(NL)^kZ(m_D).
\end{align}
Notice that
\begin{align*}
	 \mathbb{Z} \cap Z(\hat{\mu}_{\rho,D})=&~ \mathbb{Z} \cap \bigcup^{\infty}_{n=1} (NL)^n \left(\frac{\mathbb{Z} \setminus N\mathbb{Z}}{N} \cup \frac{\mathbb{Z} \setminus L\mathbb{Z}}{Lm}\right)   \\
	=&~ \mathbb{Z} \cap\left\{\left(\bigcup^{\infty}_{n=1} (NL)^n \frac{(L\mathbb{Z} \setminus NL\mathbb{Z})}{NL}\right) \cup  \left(\bigcup^{\infty}_{n=1}(NL)^n \frac{(\mathbb{Z} \setminus L\mathbb{Z})}{Lm}\right)\right\}   \\
	=&~ \left\{\mathbb{Z} \cap\left(\bigcup^{\infty}_{n=0} (NL)^n (L\mathbb{Z} \setminus NL\mathbb{Z})\right) \right\}\cup  \left\{\mathbb{Z} \cap \left(\bigcup^{\infty}_{n=1}(NL)^n \frac{(\mathbb{Z} \setminus L\mathbb{Z})}{Lm}\right)\right\}\\
	:=&~Z_1\cup Z_2.
\end{align*}
By \eqref{3.5}-\eqref{eq3.1.0} and \eqref{4.6}-\eqref{4.8}，
without loss of generality, we assume that $d=d_\kappa$. Combining  \eqref{3.7} and $l_{\kappa}=\alpha_\kappa+s_\kappa$, we have
\begin{align*}
	\tau_\kappa + \alpha_\kappa - 1 = d (\alpha_\kappa+s_\kappa) + r_\kappa \quad (0 \leq r_\kappa< \alpha_\kappa+s_\kappa).
\end{align*}

Let $\sigma = \gcd(m, (NL)^d)$. Then we can write $m = \sigma \tilde{m}$ and  $(NL)^d = \sigma \tilde{p}$ for some integers   $\tilde{m}$ and $\tilde{p}$.
Note that
$$d (\alpha_\kappa+s_\kappa)-\tau_\kappa=\alpha_\kappa - 1-r_\kappa\leq \alpha_\kappa - 1.$$
It follows that $L_\kappa^{\alpha_\kappa}\nmid \tilde{p}$ and consequently $\tilde{p}\notin Z_1$. By \eqref{4.4}, we then have $\tilde{p}\in Z_2$; that is, there exist an integer $n_0\geq 1$ and some $a\in \mathbb{Z} \setminus L\mathbb{Z}$ such that
 $(NL)^{n_0}\frac{a}{mL}=\tilde{p}$, hence
$$
(NL)^{n_0} a=mL\tilde{p}=(NL)^d L \tilde{m}.
$$
Since $a\in \mathbb{Z} \setminus L\mathbb{Z}$, the above equation shows that $n_0\geq d+1$. Hence, $N\mid\tilde{m}$, i.e.,
$N\mid\frac{m}{\gcd(m,(NL)^d)}$  and $T(\rho,D)$ is a spectral set by Theorem \ref{1.2}. This complete the proof.
\end{proof}

Using Theorem~\ref{1.2}, we can also obtain the converse that spectrality implies translational tiling for the specific class of self-similar sets under consideration.

\begin{proposition}\label{cor6.1}
Let $T(\rho,D)$ be a self-similar set defined by \eqref{1.1}, where $\rho^{-1}=NL$ and $D$ is given by \eqref{eq1.3}. If $T(\rho,D)$ is a spectral set, then $T(\rho,D)$ is a translation tile.
\end{proposition}
\begin{proof}
Write $p=\rho^{-1}=NL$. By Theorem \ref{1.2}, $T(\rho,D)$ is a spectral set implies that $N\mid\frac{m}{\gcd(m,p^d)}$  with $d$ given by \eqref{eq3.1.0}. By using Theorem $\ref{2.7}$, it suffices to show that $T(p^{-1},\tilde{p}D)$ is a translation tile if $N\mid\frac{m}{\gcd(m,p^d)}$, where
$$\tilde{p}:=L_1^{d\alpha_1-\tau_1}L_2^{d\alpha_2-\tau_2}\cdots L_b^{d\alpha_b-\tau_b}.$$

Binding together (i) and (iii) in Theorem \ref{2.8}, we need to show that
	\begin{align}\label{4.1}
		\mathbb{Z} \setminus \{0\} \subset Z(\hat{\mu}_{\rho,\tilde{p}D})=\bigcup^{\infty}_{k=1}p^kZ(m_{\tilde{p}D}).
	\end{align}
	Under the assumption that $N\mid\frac{m}{\gcd(m,p^d)}$, Proposition \ref{pro3.4} gives $d_i=d$ for all $b+1 \leq i\leq \kappa$. We further establish the following results:
	\begin{align}\label{4.3}
		\tau_i-d\alpha_i\leq 0~~~~(1\leq i\leq b)\quad {\rm{and}} \quad \tau_i-dl_i=s_i~~~~(b+1\leq i\leq \kappa).
	\end{align}
	For $1\leq i\leq b$, by using $l_i=\alpha_i$ and $d_i\leq d$, it follows that
	$$\tau_i + \alpha_i - 1 = d_i l_i + r_i\leq d\alpha_i+r_i, \quad 0\leq r_i<\alpha_i,$$
	and we thus have
	$\tau_i-d\alpha_i\leq r_i-(\alpha_i - 1)\leq 0$.
	For $b+1\leq i\leq \kappa$, the assumption $N\mid\frac{m}{\gcd(m,p^d)}$ together with  $d_i=d$ gives $$\tau_i-dl_i=r_i-(\alpha_i-1)\geq s_i,$$ i.e.,
	$r_i+1\geq \alpha_i+s_i=l_i$. Combining with $0\leq r_i< l_i$,  we must have $r_i+1=l_i$. This implies that
	$\tau_i-dl_i=r_i-(\alpha_i-1)=l_i-\alpha_i=s_i$. We complete the proof of \eqref{4.3}.

Recall that $m = L_1^{\tau_1} L_2^{\tau_2} \cdots L_\kappa^{\tau_\kappa} m'$ and $N=L_{b+1}^{s_{b+1}}\cdots L_{\kappa}^{s_{\kappa}}N'$. Let $m'=N'^tm''$ with $N'\nmid m''$ for some $t\in\mathbb{N}^*$, then it follows from $N\mid\frac{m}{\gcd(m,(NL)^d)}$ and $\gcd(L,N')=1$  that $t-d\geq 1$. Applying the results in \eqref{4.3} gives that
\begin{align*}
	\frac{p^d}{m}
	&= \frac{L_1^{d\alpha_1-\tau_1}\cdots L_b^{d\alpha_b-\tau_b}\,N'^d}
	{L_{b+1}^{\tau_{b+1}-dl_{b+1}}\cdots L_{\kappa}^{\tau_{\kappa}-dl_{\kappa}}\,m'} \\[0.4em]
	&= \frac{\tilde{p}}
	{L_{b+1}^{s_{b+1}}\cdots L_{\kappa}^{s_{\kappa}}\,N'^{\,t-d} m''} \\[0.4em]
	&= \frac{\tilde{p}}
	{N\,N'^{\,t-d-1} m''},
\end{align*}
	where $\gcd(N'^{t-d-1}m'',L)=1$. Combining with $\gcd(\tilde{p}, N)=1$, a direct calculation yields that
	\allowdisplaybreaks 
	\begin{align*}
		Z(\hat{\mu}_{\rho,\tilde{p}D})=& \bigcup^{\infty}_{n=1} p^n \left(\frac{\mathbb{Z} \setminus N\mathbb{Z}}{\tilde{p}N} \cup \frac{\mathbb{Z} \setminus L\mathbb{Z}}{\tilde{p}Lm}\right)   \\
		\supset & \left(\bigcup^{\infty}_{n=1} p^n \frac{\mathbb{Z} \setminus N\mathbb{Z}}{N}\right)\cup
		\left(\bigcup^{\infty}_{n=d+1} p^n \frac{\mathbb{Z} \setminus L\mathbb{Z}}{\tilde{p}Lm}\right)   \\
		=& \left(\bigcup^{\infty}_{n=1} p^n \frac{L\mathbb{Z} \setminus LN\mathbb{Z}}{LN}\right)\cup
		\left(\bigcup^{\infty}_{n=1} p^n \frac{\mathbb{Z} \setminus L\mathbb{Z}}{LNN'^{t-d-1}m''}\right)  \\
		\supset& \left(\bigcup^{\infty}_{n=1} p^n \frac{L\mathbb{Z} \setminus LN\mathbb{Z}}{LN}\right)\cup
		\left(\bigcup^{\infty}_{n=1} p^n \frac{\mathbb{Z} \setminus L\mathbb{Z}}{LN}\right)  \\
		=&\bigcup^{\infty}_{n=1}p^n\frac{\mathbb{Z} \setminus p\mathbb{Z}}{p}=\mathbb{Z} \setminus \{0\}.
	\end{align*}
	Consequently, \eqref{4.1} holds, and hence $T(p^{-1},\tilde{p}D)$ is a translation tile.
\end{proof}

Combining the  Propositions \ref{prop6.2} and \ref{cor6.1} leads to the equivalence between spectrality and translational tiling for the self-similar set $T(\rho, D)$. Consequently, the Fuglede's conjecture holds for these one-dimensional product-form like sets, as stated in Theorem \ref{1.3}.

\bigskip
{\bf Competing Interests:} {The authors have no conflicts of interest related to the content of this paper. }

\bibliographystyle{alpha}
\bibliography{Spectral}

@article {An2019,
	AUTHOR = {An, L. X. and Lau, K. S.},
	TITLE = {Characterization of a class of planar self-affine tile digit
	sets},
	JOURNAL = {Trans. Amer. Math. Soc.},
	FJOURNAL = {Transactions of the American Mathematical Society},
	VOLUME = {371},
	YEAR = {2019},
	NUMBER = {11},
	PAGES = {7627--7650},
	ISSN = {0002-9947,1088-6850},
	MRCLASS = {11B75 (11A63 28A80 52C22)},
	MRNUMBER = {3955530},
	MRREVIEWER = {Ben\ Joseph\ Green},
	DOI = {10.1090/tran/7481},
	URL = {https://doi.org/10.1090/tran/7481},
}

@article {An2021,
	AUTHOR = {An, L. X. and Wang, C.},
	TITLE = {On self-similar spectral measures},
	JOURNAL = {J. Funct. Anal.},
	FJOURNAL = {Journal of Functional Analysis},
	VOLUME = {280},
	YEAR = {2021},
	NUMBER = {3},
	PAGES = {Paper No. 108821, 31pp},
	ISSN = {0022-1236,1096-0783},
	MRCLASS = {42B10 (28A80 42A85)},
	MRNUMBER = {4170789},
	MRREVIEWER = {Jinjun\ Li},
	DOI = {10.1016/j.jfa.2020.108821},
	URL = {https://doi.org/10.1016/j.jfa.2020.108821},
}

@article {AL_2023,
	AUTHOR = {An, L. X. and Lai, C. K.},
	TITLE = {Arbitrarily sparse spectra for self-affine spectral measures.},
	JOURNAL = {Analysis Math.},
	FJOURNAL = {Analysis Mathematics},
	VOLUME = {49},
	YEAR = {2023},
	PAGES = {19--42},
	ISSN = {},
	MRCLASS = {},
	MRNUMBER = {},
	MRREVIEWER = {},
	DOI = {},
	URL = {},
}

@article {An-He-Lai2023,
	AUTHOR = {An, L. X. and He, X. G. and Lai, C. K.},
	TITLE = {Classification of spectral self-similar measures with
	four-digit elements},
	JOURNAL = {Asian J. Math.},
	FJOURNAL = {Asian Journal of Mathematics},
	VOLUME = {27},
	YEAR = {2023},
	NUMBER = {4},
	PAGES = {467--492},
	ISSN = {1093-6106,1945-0036},
	MRCLASS = {42B10 (28A80 42C30)},
	MRNUMBER = {4775556},
	DOI = {10.4310/ajm.2023.v27.n4.a2},
	URL = {https://doi.org/10.4310/ajm.2023.v27.n4.a2},
}

@article {Chen-Liu-Zheng2024,
	AUTHOR = {Chen, M. L. and Liu, J. C. and Zheng, J.},
	TITLE = {Tiling and spectrality for generalized {S}ierpinski
	self-affine sets},
	JOURNAL = {J. Geom. Anal.},
	FJOURNAL = {Journal of Geometric Analysis},
	VOLUME = {34},
	YEAR = {2024},
	NUMBER = {1},
	PAGES = {Paper No. 5, 32pp},
	ISSN = {1050-6926,1559-002X},
	MRCLASS = {28A25 (28A80 42C05 46C05)},
	MRNUMBER = {4658596},
	MRREVIEWER = {Richard\ Becker},
	DOI = {10.1007/s12220-023-01447-y},
	URL = {https://doi.org/10.1007/s12220-023-01447-y},
}

@article {CM,
	AUTHOR = {Coven, E. M. and Meyerowitz, A.},
	TITLE = {Tiling the integers with translates of one finite set},
	JOURNAL = {J. Algebra},
	FJOURNAL = {Journal of Algebra},
	VOLUME = {212},
	YEAR = {1999},
	NUMBER = {1},
	PAGES = {161--174},
	ISSN = {0021-8693,1090-266X},
	MRCLASS = {11B75},
	MRNUMBER = {1670646},
	MRREVIEWER = {Mihail\ N.\ Kolountzakis},
	DOI = {10.1006/jabr.1998.7628},
	URL = {https://doi.org/10.1006/jabr.1998.7628},
}

@article {Dai2012,
	AUTHOR = {Dai, X. R.},
	TITLE = {When does a {B}ernoulli convolution admit a spectrum?},
	JOURNAL = {Adv. Math.},
	FJOURNAL = {Advances in Mathematics},
	VOLUME = {231},
	YEAR = {2012},
	NUMBER = {3-4},
	PAGES = {1681--1693},
	ISSN = {0001-8708,1090-2082},
	MRCLASS = {42B05 (28A78 28A80 42A16 42A65 42C30)},
	MRNUMBER = {2964620},
	MRREVIEWER = {Walter\ Schempp},
	DOI = {10.1016/j.aim.2012.06.026},
	URL = {https://doi.org/10.1016/j.aim.2012.06.026},
}

@article {Dai-He-Lau_2014,
	AUTHOR = {Dai, X. R. and He, X. G. and Lau, K. S.},
	TITLE = {On spectral {$N$}-{B}ernoulli measures},
	JOURNAL = {Adv. Math.},
	FJOURNAL = {Advances in Mathematics},
	VOLUME = {259},
	YEAR = {2014},
	PAGES = {511--531},
	ISSN = {0001-8708,1090-2082},
	MRCLASS = {28A80 (42C05)},
	MRNUMBER = {3197665},
	MRREVIEWER = {Charlene\ Kalle},
	DOI = {10.1016/j.aim.2014.03.026},
	URL = {https://doi.org/10.1016/j.aim.2014.03.026},
}

@article {Dut-Hau-Lai2019,
	AUTHOR = {Dutkay, D. E. and Haussermann, J. and Lai, C. K.},
	TITLE = {Hadamard triples generate self-affine spectral measures},
	JOURNAL = {Trans. Amer. Math. Soc.},
	FJOURNAL = {Transactions of the American Mathematical Society},
	VOLUME = {371},
	YEAR = {2019},
	NUMBER = {2},
	PAGES = {1439--1481},
	ISSN = {0002-9947,1088-6850},
	MRCLASS = {42B05 (28A25 42A85)},
	MRNUMBER = {3885185},
	MRREVIEWER = {Xing-Gang\ He},
	DOI = {10.1090/tran/7325},
	URL = {https://doi.org/10.1090/tran/7325},
}

@article {DJ_2007,
	AUTHOR = {Dutkay, D. E. and Jorgensen, P. E. T.},
	TITLE = {Fourier frequencies in affine iterated function systems},
	JOURNAL = {J. Funct. Anal.},
	FJOURNAL = {Journal of Functional Analysis},
	VOLUME = {247},
	YEAR = {2007},
	NUMBER = {1},
	PAGES = {110--137},
	ISSN = {0022-1236,1096-0783},
	MRCLASS = {42B05 (28A80 37A99 37B10)},
	MRNUMBER = {2319756},
	MRREVIEWER = {St\'ephane\ Seuret},
	DOI = {10.1016/j.jfa.2007.03.002},
	URL = {https://doi.org/10.1016/j.jfa.2007.03.002},
}

@article {DL_2014,
	AUTHOR = {Dutkay, D. E. and Lai, C. K.},
	TITLE = {Uniformity of measures with {F}ourier frames},
	JOURNAL = {Adv. Math.},
	FJOURNAL = {Advances in Mathematics},
	VOLUME = {252},
	YEAR = {2014},
	PAGES = {684--707},
	ISSN = {0001-8708,1090-2082},
	MRCLASS = {28A80 (42C15)},
	MRNUMBER = {3144246},
	MRREVIEWER = {Hua\ Qiu},
	DOI = {10.1016/j.aim.2013.11.012},
	URL = {https://doi.org/10.1016/j.aim.2013.11.012},
}

@article {DJ_2009,
	AUTHOR = {Dutkay, D. E. and Jorgensen, P. E. T.},
	TITLE = {Probability and {F}ourier duality for affine iterated function
	systems},
	JOURNAL = {Acta Appl. Math.},
	FJOURNAL = {Acta Applicandae Mathematicae},
	VOLUME = {107},
	YEAR = {2009},
	NUMBER = {1-3},
	PAGES = {293--311},
	ISSN = {0167-8019,1572-9036},
	MRCLASS = {37A50 (28C15 42B35 42C05)},
	MRNUMBER = {2520021},
	MRREVIEWER = {Jan-Olav\ R\"onning},
	DOI = {10.1007/s10440-008-9384-2},
	URL = {https://doi.org/10.1007/s10440-008-9384-2},
}

@article {DJ_2009_2,
	AUTHOR = {Dutkay, D. E. and Jorgensen, P. E. T.},
	TITLE = {Quasiperiodic spectra and orthogonality for iterated function
	system measures},
	JOURNAL = {Math. Z.},
	FJOURNAL = {Mathematische Zeitschrift},
	VOLUME = {261},
	YEAR = {2009},
	NUMBER = {2},
	PAGES = {373--397},
	ISSN = {0025-5874,1432-1823},
	MRCLASS = {28A80 (42C05)},
	MRNUMBER = {2457304},
	MRREVIEWER = {Kasso\ A.\ Okoudjou},
	DOI = {10.1007/s00209-008-0329-2},
	URL = {https://doi.org/10.1007/s00209-008-0329-2},
}

@article {Fan2019,
	AUTHOR = {Fan, A. H.  and Fan, S. L. and Liao, L. M. and Shi, R. X.},
	TITLE = {Fuglede's conjecture holds in {$\Bbb {Q}_p$}},
	JOURNAL = {Math. Ann.},
	FJOURNAL = {Mathematische Annalen},
	VOLUME = {375},
	YEAR = {2019},
	NUMBER = {1-2},
	PAGES = {315--341},
	ISSN = {0025-5831,1432-1807},
	MRCLASS = {43A99 (05B45 26E30)},
	MRNUMBER = {4000244},
	DOI = {10.1007/s00208-019-01867-8},
	URL = {https://doi.org/10.1007/s00208-019-01867-8},

}

@book {Falconer_1990,
	AUTHOR = {Falconer, K.},
	TITLE = {Fractal geometry},
	NOTE = {Mathematical foundations and applications},
	PUBLISHER = {John Wiley \& Sons, Ltd., Chichester},
	YEAR = {1990},
	PAGES = {xxii+288},
	ISBN = {0-471-92287-0},
	MRCLASS = {28A80 (00A69 11K55 28-01 58F13 60G18)},
	MRNUMBER = {1102677},
	MRREVIEWER = {Christoph\ Bandt},
}

@article {Fu-He-Lau2015,
	AUTHOR = {Fu, X. Y. and He, X. G. and Lau, K. S.},
	TITLE = {Spectrality of self-similar tiles},
	JOURNAL = {Constr. Approx.},
	FJOURNAL = {Constructive Approximation. An International Journal for
	Approximations and Expansions},
	VOLUME = {42},
	YEAR = {2015},
	NUMBER = {3},
	PAGES = {519--541},
	ISSN = {0176-4276,1432-0940},
	MRCLASS = {42C15 (28A80)},
	MRNUMBER = {3416166},
	MRREVIEWER = {Chun-Kit\ Lai},
	DOI = {10.1007/s00365-015-9306-2},
	URL = {https://doi.org/10.1007/s00365-015-9306-2},
}

@article {Fuglede1974,
	AUTHOR = {Fuglede, B.},
	TITLE = {Commuting self-adjoint partial differential operators and a
	group theoretic problem},
	JOURNAL = {J. Funct. Anal.},
	FJOURNAL = {Journal of Functional Analysis},
	VOLUME = {16},
	YEAR = {1974},
	PAGES = {101--121},
	ISSN = {0022-1236},
	MRCLASS = {47F05 (81.47)},
	MRNUMBER = {470754},
	DOI = {10.1016/0022-1236(74)90072-x},
	URL = {https://doi.org/10.1016/0022-1236(74)90072-x},
}

@article {HLL_2013,
	AUTHOR = {He, X. G.  and Lai, C. K. and Lau, K. S.},
	TITLE = {Exponential spectra in {$L^2(\mu)$}},
	JOURNAL = {Appl. Comput. Harmon. Anal.},
	FJOURNAL = {Applied and Computational Harmonic Analysis. Time-Frequency
	and Time-Scale Analysis, Wavelets, Numerical Algorithms, and
	Applications},
	VOLUME = {34},
	YEAR = {2013},
	NUMBER = {3},
	PAGES = {327--338},
	ISSN = {1063-5203,1096-603X},
	MRCLASS = {42C15 (46E30)},
	MRNUMBER = {3027906},
	MRREVIEWER = {Keri\ A.\ Kornelson},
	DOI = {10.1016/j.acha.2012.05.003},
	URL = {https://doi.org/10.1016/j.acha.2012.05.003},
}

@article{IMP_2017,
    AUTHOR = {Iosevich, A. and Mayeli, A. and Pakianathan, J.},
    TITLE = {The {F}uglede conjecture holds in $\mathbb{Z}_p \times \mathbb{Z}_p$},
    JOURNAL = {Anal. PDE},
    FJOURNAL = {Analysis \& PDE},
    VOLUME = {10},
    YEAR = {2017},
    NUMBER = {4},
    PAGES = {757--764},
    ISSN = {2157-5045,1948-206X},
	MRCLASS = {11T30 (05A18 41A10 42B05 52C20)},
	MRNUMBER = {3649367},
	MRREVIEWER = {B\'ela\ Uhrin},
	DOI = {10.2140/apde.2017.10.757},
	URL = {https://doi.org/10.2140/apde.2017.10.757},
}

@article {Jorgensen-Pedersen_1998,
	AUTHOR = {Jorgensen, P. E. T. and Pedersen, S.},
	TITLE = {Dense analytic subspaces in fractal {$L^2$}-spaces},
	JOURNAL = {J. Anal. Math.},
	FJOURNAL = {Journal d'Analyse Math\'ematique},
	VOLUME = {75},
	YEAR = {1998},
	PAGES = {185--228},
	ISSN = {0021-7670,1565-8538},
	MRCLASS = {46E30 (28A75 42C05 46L55 47B38)},
	MRNUMBER = {1655831},
	MRREVIEWER = {Javier\ Soria},
	DOI = {10.1007/BF02788699},
	URL = {https://doi.org/10.1007/BF02788699},
}

@incollection {Ken1992,
	AUTHOR = {Kenyon, R.},
	TITLE = {Self-replicating tilings},
	BOOKTITLE = {Symbolic dynamics and its applications ({N}ew {H}aven, {CT},
	1991)},
	SERIES = {Contemp. Math.},
	VOLUME = {135},
	PAGES = {239--263},
	PUBLISHER = {Amer. Math. Soc., Providence, RI},
	YEAR = {1992},
	ISBN = {0-8218-5146-2},
	MRCLASS = {52C22 (05B45)},
	MRNUMBER = {1185093},
	MRREVIEWER = {Carl\ E.\ Linderholm},
	DOI = {10.1090/conm/135/1185093},
	URL = {https://doi.org/10.1090/conm/135/1185093},
}

@article {Laba2001,
	AUTHOR = {{\L}aba, I.},
	TITLE = {Fuglede's conjecture for a union of two intervals},
	JOURNAL = {Proc. Amer. Math. Soc.},
	FJOURNAL = {Proceedings of the American Mathematical Society},
	VOLUME = {129},
	YEAR = {2001},
	NUMBER = {10},
	PAGES = {2965--2972},
	ISSN = {0002-9939,1088-6826},
	MRCLASS = {42A99},
	MRNUMBER = {1840101},
	MRREVIEWER = {Steen\ Pedersen},
	DOI = {10.1090/S0002-9939-01-06035-X},
	URL = {https://doi.org/10.1090/S0002-9939-01-06035-X},
}

@article {Lai2017,
	AUTHOR = {Lai, C. K. and Lau, K. S. and Rao, H.},
	TITLE = {Classification of tile digit sets as product-forms},
	JOURNAL = {Trans. Amer. Math. Soc.},
	FJOURNAL = {Transactions of the American Mathematical Society},
	VOLUME = {369},
	YEAR = {2017},
	NUMBER = {1},
	PAGES = {623--644},
	ISSN = {0002-9947,1088-6850},
	MRCLASS = {52C22 (11A63 11B75 28A80)},
	MRNUMBER = {3557788},
	MRREVIEWER = {Benoit\ Loridant},
	DOI = {10.1090/tran/6703},
	URL = {https://doi.org/10.1090/tran/6703},
}

@article {Lau2008,
	AUTHOR = {Hu, T. Y. and Lau, K. S.},
	TITLE = {Spectral property of the {B}ernoulli convolutions},
	JOURNAL = {Adv. Math.},
	FJOURNAL = {Advances in Mathematics},
	VOLUME = {219},
	YEAR = {2008},
	NUMBER = {2},
	PAGES = {554--567},
	ISSN = {0001-8708,1090-2082},
	MRCLASS = {42C05 (28A78 28A80 42A65)},
	MRNUMBER = {2435649},
	MRREVIEWER = {Sze-Man\ Ngai},
	DOI = {10.1016/j.aim.2008.05.004},
	URL = {https://doi.org/10.1016/j.aim.2008.05.004},
}

@article {LW_1996,
	AUTHOR = {Lagarias, J. C. and Wang, Y.},
	TITLE = {Self-affine tiles in {$\mathbb{R}^n$}},
	JOURNAL = {Adv. Math.},
	FJOURNAL = {Advances in Mathematics},
	VOLUME = {121},
	YEAR = {1996},
	NUMBER = {1},
	PAGES = {21--49},
	ISSN = {0001-8708,1090-2082},
	MRCLASS = {52C22 (11A63)},
	MRNUMBER = {1399601},
	MRREVIEWER = {Richard\ Kenyon},
	DOI = {10.1006/aima.1996.0045},
	URL = {https://doi.org/10.1006/aima.1996.0045},
}

@article {LW_1996_2,
	AUTHOR = {Lagarias, J. C. and Wang, Y.},
	TITLE = {Integral self-affine tiles in {$\mathbb{R}^n$}. {I}. {S}tandard
	and nonstandard digit sets},
	JOURNAL = {J. London Math. Soc. (2)},
	FJOURNAL = {Journal of the London Mathematical Society. Second Series},
	VOLUME = {54},
	YEAR = {1996},
	NUMBER = {1},
	PAGES = {161--179},
	ISSN = {0024-6107,1469-7750},
	MRCLASS = {52C22 (39B52)},
	MRNUMBER = {1395075},
	MRREVIEWER = {Mihail\ N.\ Kolountzakis},
	DOI = {10.1112/jlms/54.1.161},
	URL = {https://doi.org/10.1112/jlms/54.1.161},
}

@article {Lag1997,
    AUTHOR = {Lagarias, J. C. and Wang, Y.},
     TITLE = {Integral self-affine tiles in {$\mathbb{R}^n$}. {II}. {L}attice
              tilings},
   JOURNAL = {J. Fourier Anal. Appl.},
  FJOURNAL = {The Journal of Fourier Analysis and Applications},
    VOLUME = {3},
      YEAR = {1997},
    NUMBER = {1},
     PAGES = {83--102},
      ISSN = {1069-5869,1531-5851},
   MRCLASS = {52C22},
  MRNUMBER = {1428817},
MRREVIEWER = {Mihail\ N.\ Kolountzakis},
       DOI = {10.1007/s00041-001-4051-2},
       URL = {https://doi.org/10.1007/s00041-001-4051-2},
}

@article {Lev2022,
	AUTHOR = {Lev, N. and Matolcsi, M.},
	TITLE = {The {F}uglede conjecture for convex domains is true in all
	dimensions},
	JOURNAL = {Acta Math.},
	FJOURNAL = {Acta Mathematica},
	VOLUME = {228},
	YEAR = {2022},
	NUMBER = {2},
	PAGES = {385--420},
	ISSN = {0001-5962,1871-2509},
	MRCLASS = {52A05 (28A75 42B10)},
	MRNUMBER = {4448683},
	MRREVIEWER = {Grigory\ M.\ Ivanov},
	DOI = {10.4310/acta.2022.v228.n2.a3},
	URL = {https://doi.org/10.4310/acta.2022.v228.n2.a3},
}

@article {LaW_2002,
	AUTHOR = {{\L}aba, I. and Wang, Y.},
	TITLE = {On spectral {C}antor measures},
	JOURNAL = {J. Funct. Anal.},
	FJOURNAL = {Journal of Functional Analysis},
	VOLUME = {193},
	YEAR = {2002},
	NUMBER = {2},
	PAGES = {409--420},
	ISSN = {0022-1236,1096-0783},
	MRCLASS = {28A80},
	MRNUMBER = {1929508},
	MRREVIEWER = {Henning\ Fernau},
	DOI = {10.1006/jfan.2001.3941},
	URL = {https://doi.org/10.1006/jfan.2001.3941},
}

@article {LR_2025,
	AUTHOR = {Li, Q. and Rao, H.},
	TITLE = {Characterization of self-affine tile digit sets on {$\Bbb
	R^n$}},
	JOURNAL = {Nonlinearity},
	FJOURNAL = {Nonlinearity},
	VOLUME = {38},
	YEAR = {2025},
	NUMBER = {9},
	PAGES = {Paper No. 095011, 14pp},
	ISSN = {0951-7715,1361-6544},
	MRCLASS = {52C23 (28A80)},
	MRNUMBER = {4957830},
	DOI = {10.1088/1361-6544/adffdd},
	URL = {https://doi.org/10.1088/1361-6544/adffdd},
}

@article {Li2024,
    AUTHOR = {Li, W. X. and Miao, J. J. and Wang, Z. Q.},
     TITLE = {Spectrality of infinite convolutions and random convolutions},
   JOURNAL = {J. Funct. Anal.},
  FJOURNAL = {Journal of Functional Analysis},
    VOLUME = {287},
      YEAR = {2024},
    NUMBER = {7},
     PAGES = {Paper No. 110539, 35pp},
      ISSN = {0022-1236,1096-0783},
   MRCLASS = {42C30 (28A80)},
  MRNUMBER = {4760209},
MRREVIEWER = {Zhi-Yi\ Wu},
       DOI = {10.1016/j.jfa.2024.110539},
       URL = {https://doi.org/10.1016/j.jfa.2024.110539},
}

@article {Liu-Peng-Wu_2019,
	AUTHOR = {Liu, J. C. and Peng, R. G. and Wu, H. H.},
	TITLE = {Spectral properties of self-similar measures with product-form
	digit sets},
	JOURNAL = {J. Math. Anal. Appl.},
	FJOURNAL = {Journal of Mathematical Analysis and Applications},
	VOLUME = {473},
	YEAR = {2019},
	NUMBER = {1},
	PAGES = {479--489},
	ISSN = {0022-247X,1096-0813},
	MRCLASS = {28A80},
	MRNUMBER = {3912833},
	MRREVIEWER = {Paul\ Surer},
	DOI = {10.1016/j.jmaa.2018.12.062},
	URL = {https://doi.org/10.1016/j.jmaa.2018.12.062},
}

@article {Matolcsi2005,
	AUTHOR = {Matolcsi, M.},
	TITLE = {Fuglede's conjecture fails in dimension 4},
	JOURNAL = {Proc. Amer. Math. Soc.},
	FJOURNAL = {Proceedings of the American Mathematical Society},
	VOLUME = {133},
	YEAR = {2005},
	NUMBER = {10},
	PAGES = {3021--3026},
	ISSN = {0002-9939,1088-6826},
	MRCLASS = {11H31 (42B05 52C22)},
	MRNUMBER = {2159781},
	MRREVIEWER = {B\'ela\ Uhrin},
	DOI = {10.1090/S0002-9939-05-07874-3},
	URL = {https://doi.org/10.1090/S0002-9939-05-07874-3},
}

@article {Shi_2019,
	AUTHOR = {Shi, R. X.},
	TITLE = {Fuglede's conjecture holds on cyclic groups {$\Bbb Z_{pqr}$}},
	JOURNAL = {Discrete Anal.},
	FJOURNAL = {Discrete Analysis},
    PAGES = {Paper No. 14, 14pp},
	YEAR = {2019},
	ISSN = {2397-3129},
	MRCLASS = {42B10 (11B75 43A70)},
	MRNUMBER = {4025286},
	MRREVIEWER = {Li-Xiang\ An},
	DOI = {10.19086/da},
	URL = {https://doi.org/10.19086/da},
}

@article {Strichartz_2000,
	AUTHOR = {Strichartz, R. S.},
	TITLE = {Mock {F}ourier series and transforms associated with certain
	{C}antor measures},
	JOURNAL = {J. Anal. Math.},
	FJOURNAL = {Journal d'Analyse Math\'ematique},
	VOLUME = {81},
	YEAR = {2000},
	PAGES = {209--238},
	ISSN = {0021-7670,1565-8538},
	MRCLASS = {42A38 (28A75 42C05)},
	MRNUMBER = {1785282},
	MRREVIEWER = {Steen\ Pedersen},
	DOI = {10.1007/BF02788990},
	URL = {https://doi.org/10.1007/BF02788990},
}

@article {Strichartz_2006,
	AUTHOR = {Strichartz, R. S.},
	TITLE = {Convergence of Mock {F}ourier series},
	JOURNAL = {J. Anal. Math.},
	FJOURNAL = {Journal d'Analyse Math\'ematique},
	VOLUME = {99},
	YEAR = {2006},
	PAGES = {333--353},
	ISSN = {0021-7670,1565-8538},
	MRCLASS = {},
	MRNUMBER = {},
	MRREVIEWER = {},
	DOI = {},
	URL = {},
}

@article {Tao2003,
	AUTHOR = {Tao, T.},
	TITLE = {Fuglede's conjecture is false in 5 and higher dimensions},
	JOURNAL = {Math. Res. Lett.},
	FJOURNAL = {Mathematical Research Letters},
	VOLUME = {11},
	YEAR = {2004},
	NUMBER = {2-3},
	PAGES = {251--258},
	ISSN = {1073-2780},
	MRCLASS = {42B99 (43A45 46E30 52C22)},
	MRNUMBER = {2067470},
	MRREVIEWER = {B\'ela\ Uhrin},
	DOI = {10.4310/MRL.2004.v11.n2.a8},
	URL = {https://doi.org/10.4310/MRL.2004.v11.n2.a8},
}

@article {Wu2024,
	AUTHOR = {Wu, H. H.},
	TITLE = {Spectral self-similar measures with alternate contraction
	ratios and consecutive digits},
	JOURNAL = {Adv. Math.},
	FJOURNAL = {Advances in Mathematics},
	VOLUME = {443},
	YEAR = {2024},
	PAGES = {Paper No. 109585, 33pp},
	ISSN = {0001-8708,1090-2082},
	MRCLASS = {28A80 (42B10 42C05)},
	MRNUMBER = {4712262},
	MRREVIEWER = {Ming-Liang\ Chen},
	DOI = {10.1016/j.aim.2024.109585},
	URL = {https://doi.org/10.1016/j.aim.2024.109585},
}

\end{document}